\documentclass[leqno]{amsart}
\usepackage{dsfont}
\usepackage{cases}
\usepackage{color}
\usepackage{ulem}
\usepackage{framed}

\newtheorem{theorem}{Theorem}[section]
\newtheorem{proposition}[theorem]{Proposition}
\newtheorem{lemma}[theorem]{Lemma}
\newtheorem{corollary}[theorem]{Corollary}
\theoremstyle{definition}
\newtheorem{definition}[theorem]{Definition}

\theoremstyle{remark}
\newtheorem{remark}[theorem]{Remark}
 
\numberwithin{equation}{section}
\newcommand{\argmax}{\mathop{\rm arg~max}\limits}

\begin{document}
\title[On viscosity solutions of space-fractional diffusion equations]{On viscosity solutions of space-fractional diffusion equations of Caputo type}
\author{Tokinaga~Namba}
\address{Mathematical Science \& Technology Research Laboratory, Advanced Technology Research Laboratories, Research \& Development, Nippon Steel Corporation, 20-1 Shintomi, Futtsu, Chiba 293-8511, Japan.}
\email{namba.rb3.tokinaga@jp.nipponsteel.com}
\author{Piotr~Rybka}
\address{Institute of Applied Mathematics and Mechanics, Warsaw University ul. Banacha 2, 02-097 Warsaw, Poland}
\email{rybka@mimuw.edu.pl}

\begin{abstract}
We study a fractional diffusion problem in the divergence form in one space dimension. We define a notion of the viscosity solution. We prove existence of viscosity solutions to the fractional diffusion problem with the Dirichlet boundary values by Perron's method. Their uniqueness follows from a proper maximum principle. We also show a stability result and basic regularity of solutions.
\end{abstract}

\maketitle
\normalem

\section{Introduction}
In this paper we study the Cauchy-Dirichlet problem whose governing equation is 
\begin{equation}\label{e:rn1}
u_t = (D^\alpha_x u)_x+f \quad \text{in $(0,l)\times(0,T)$},
\end{equation}
where $f$ is a continuous function.
The operator $D^\alpha_x u$ is the spacial Caputo derivative and, for an absolutely continuous function $v$, it is defined as
\begin{equation}\label{rdC}
D^\alpha_x v(x) = \frac{1}{\Gamma(1-\alpha)} \int_0^x \frac {v'(y)} {(x-y)^\alpha}dy.
\end{equation}
Our main goal is to introduce a notion of generalized solutions and to show that this notion leads to the well-posedness of the Cauchy-Dirichlet problem of \eqref{e:rn1}.

Before we engage into a theoretical development we will explain our motivation to study such an equation.
In fact, problem \eqref{e:rn1} is a simplification of a free boundary problem (FBP), which should be considered on a noncylindrical domain.
It is a simplification of a model of the sub-surface movement of water, under the assumption that it moves as a saturated `plug' through a soil that has a constant moisture storage capacity, see \cite{Voller}.
From the point of view of modeling equation \eqref{e:rn1} should be understood as a balance equation, where we have on the right hand side the divergence of a flux. This would be particularly important if we were study the multidimensional case.
However, we will restrict our attention only to the one-dimensional case, because there is plenty of open question.

In particular, it seems that the literature on well-posedness of linear problems of the form like  \eqref{e:rn1} is rather limited.
However, the number  of papers and books on time fractional problems is growing, so we mention just a few monographs, \cite{Diethelm}, \cite{Podlubny} or \cite{Zhou}.
A reason for such a situation is that this line of research originated from the Volterra integral equations, see e.g. \cite{Pruss}. 

There is number of recent papers addressing the fractional diffusion problems from the point of view of the semigroup theory, see \cite{Baeumer1},  \cite{Baeumer2}. These authors construct a strongly continuous semigroup.
A stronger result in this direction is obtained in \cite{Rys}, where the author shows that operator  $(D^\alpha_x u)_x$ generates an analytic semigroup, but with different boundary conditions in comparison with ours.

It is obvious that before we tackle equation \eqref{e:rn1} in a domain which changes in time, we must understand it in a fixed domain.
Our paper is just a step toward this bigger goal.
We notice that the literature on time fractional equations, which includes several books, is quite broad and we will use the applicable tools.

Among the many possible methods to address \eqref{e:rn1} we choose the theory of viscosity solutions.
Since the setting which we consider in this paper apparently has not been considered in the literature, we have to find a suitable notion of viscosity solutions.
We note that there are papers dealing with viscosity solutions for integro-differential equations, \cite{BarlesChasseigneCiomagaImbert}, \cite{GigaNamba} or \cite{Mou}, but the non-locality is with respect to the time variable, not space variable like in our case.

A special feature of \eqref{e:rn1} is that formally we require existence of two spacial derivatives.
However, we come up with a seemingly less demanding definition of solution.
We devote Section 2 to this issue.
We also present our definition of the viscosity solution, which draws heavily on our experience with time fractional Hamilton-Jacobi equations, \cite{GigaNamba}.

After settling the issue of a proper definition of solution, we establish its basic properties.
The main result is a  comparison principle, see Theorem \ref{t:comparison}.
Once we show it, we may establish existence of solutions by means of the Perron method. 
This is done in Theorem \ref{t:existence}.
The main difficulty is showing that the sup of subsolutions is also a supersolution. 
The comparison principle implies uniqueness of solutions. 
We also show the stability of solutions with respect to the fractional order of equation, this is the content of Theorem \ref{t:staborder}. 

Finally, we address regularity of solutions in Section 7. Namely, we show that if the data are Lipschitz continuous, then the unique solution is H\"older continuous with respect to the space variable, with exponent $\alpha$, see Proposition \ref{p:lateralregularity}. Moreover, this solution is Lipschitz continuous with respect to the time variable, see  Propositions \ref{p:initial} and  \ref{p:finalT} and locally Lipschitz continuous is space, cf. Proposition \ref{p:regularity}.

\section{Preliminaries on fractional derivatives}

Throughout this paper the integral of the form
$$
I_{(a,b)}[w](x)=\int_a^bw(x,z)\frac{dz}{z^{\alpha+2}},
$$
where $0\le a<b<+\infty$, is mainly handled for semicontinuous functions $w:[0,l)\times[0,l)\to\mathbb{R}$.
In this section we present properties of such a function derived from $(D_x^\alpha u)_x$, which is introduced as an operator $K$ later.
Here, all integrals are interpreted in terms of the Lebesgue integral.
When $a>0$, we say that $I_{(a,b)}[w](x)$ \emph{is well-defined} if either $I_{(a,b)}[w^+](x)$ or $I_{(a,b)}[w^-](x)$ is finite, where $w^\pm:=\max\{\pm w, 0\}$.
When $a=0$, $I_{(0,b)}[w](x)$ is regarded as $\lim_{a\to0^+}I_{(a,b)}[w](x)$, and we say that $I_{(0,b)}[w](x)$ is \emph{well-defined} if $I_{(a,b)}[w^\pm](x)$ are finite for each $a\in(0,b)$ and $\lim_{a\to0^+}I_{(a,b)}[w](x)$ exists and it is finite.
In other words, these conditions are equivalent to integrability of $z\mapsto w(x,z)$ for a.e. $x\in (0,l)$.

The following proposition plays an essential role in defining the solution introduced in this paper.
\begin{proposition}\label{p:another}
Let $u:[0,l)\to\mathbb{R}$ be such that $u\in C^2(0,l)\cap C[0,l)$ and $u'\in L^1(0,l)$.
Then, $(D^\alpha_x u)_x$ exists everywhere in $(0,l)$ and
\begin{equation}\label{e:another}
\begin{split}
(D^\alpha u)_x(x)=
\frac{1}{\Gamma(1-\alpha)}&\left(\frac{\alpha(u(0)-u(x))+(\alpha+1)u'(x)x}{x^{\alpha+1}}\right.\\
&\quad\left.+\alpha(\alpha+1)\int_0^x[u(x-z)-u(x)+u'(x)z]\frac{dz}{z^{\alpha+2}}\right)
\end{split}
\end{equation}
for $x\in(0,l)$.
\end{proposition}

\begin{proof}
First of all, we note that for each $x\in(0,l)$ the integral on the right-hand side of \eqref{e:another} exists, that is, $[u(x-z)-u(x)+u'(x)z]/z^{\alpha+2}\in L^1(0,x)$.
In fact, Taylor's theorem implies that for each $z\in(0,x/2)$
$$
|u(x-z)-u(x)+u'(x)z|\le \frac{\sup_{y\in(x-z,x)}|u''(y)|}{2}z^2\le\frac{\max_{[x/2,x]}|u''|}{2}z^2.
$$
The right-hand side multiplied by $z^{-\alpha-2}$ is integrable on $(0,x/2)$ and so the dominated convergence theorem implies that 
$$
\text{$\displaystyle\int_{0}^{x/2}[u(x-z)-u(x)+u'(x)z]\frac{dz}{z^{\alpha+2}}$\, exists.}
$$
The integral on the remaining interval $(x/2,x)$ does not present a problem since $z\mapsto (u(x-z)-u(x)+u'(x)z)z^{-\alpha-2}$ is bounded on $(x/2,x)$ and hence integrable.
Thus, it is enough to show \eqref{e:another} to prove Proposition \ref{p:another}.
Hereafter we show \eqref{e:another}.
To this end, we fix any small $\varepsilon>0$ and we introduce
$$
I_1(x):=\int_0^\varepsilon\frac{u'(y)}{(x-y)^\alpha}dy\quad\text{and}\quad I_2(x):=\int_\varepsilon^x\frac{u'(y)}{(x-y)^\alpha}dy
$$
for $x\in(2\varepsilon,l)$. We note that $(D_x^\alpha u)_x=\frac{1}{\Gamma(1-\alpha)}\frac{d}{dx}(I_1+I_2)$.

It is easily seen that the differentiation under the integral sign is applicable to $I_1$ for each $x\in(2\varepsilon,l)$, since $1/(x-y)^{\alpha+1}\le 1/\varepsilon^{\alpha+1}$ for all $x\in(2\varepsilon,l)$ and $y\in(0,\varepsilon)$ and $u'$ is integrable.
Together with straighforward calculations we get
\begin{align*}
\frac{d}{dx}I_1(x)
&=-\alpha\int_0^\varepsilon\frac{u'(y)}{(x-y)^{\alpha+1}}dy\\
&=\alpha\left(\frac{u(0)}{x^{\alpha+1}}-\frac{u(\varepsilon)}{(x-\varepsilon)^{\alpha+1}}\right)+\alpha(\alpha+1)\int_0^\varepsilon\frac{u(y)}{(x-y)^{\alpha+2}}dy\\
&=\frac{\alpha(u(0)-u(x))+(\alpha+1)u'(x)x}{x^{\alpha+1}}\\
&\quad-\frac{\alpha(u(\varepsilon)-u(x))+(\alpha+1)u'(x)(x-\varepsilon)}{(x-\varepsilon)^{\alpha+1}}\\
&\quad+\alpha(\alpha+1)\int_0^\varepsilon\frac{u(y)-u(x)-u'(x)(y-x)}{(x-y)^{\alpha+2}}dy.
\end{align*}

Differentiation under the integral sign is applicable also to $I_2$  (see, e.g., \cite[subsection 7.2.1]{GigaGigaSaal}), and then
$$
\frac{d}{dx}I_2(x)=\frac{u'(\varepsilon)}{(x-\varepsilon)^\alpha}+\int_\varepsilon^{x}\frac{u''(y)}{(x-y)^\alpha}dy.
$$
Since $u\in C^2(0,l)$, it holds that
$$
\lim_{y\to x}\frac{u'(y)-u'(x)}{(x-y)^{\alpha}}=0
$$
and
$$
\lim_{y\to x}\frac{u(y)-u(x)-u'(x)(y-x)}{(x-y)^{\alpha+1}}=0.
$$
Thus, straightforward calculations imply that
\begin{align*}
\frac{d}{dx}I_2(x)
&=\frac{u'(\varepsilon)}{(x-\varepsilon)^\alpha}+\int_\varepsilon^{x}\frac{\frac{d}{dy}(u'(y)-u'(x))}{(x-y)^\alpha}dy\\
&=\frac{u'(x)}{(x-\varepsilon)^\alpha}-\alpha\int_\varepsilon^{x}\frac{u'(y)-u'(x)}{(x-y)^{\alpha+1}}dy\\
&=\frac{u'(x)}{(x-\varepsilon)^\alpha}-\alpha\int_\varepsilon^{x}\frac{\frac{d}{dy}(u(y)-u(x)-u'(x)(y-x))}{(x-y)^{\alpha+1}}dy\\
&=\frac{\alpha(u(\varepsilon)-u(x))+(\alpha+1)u'(x)(x-\varepsilon)}{(x-\varepsilon)^{\alpha+1}}\\
&\quad+\alpha(\alpha+1)\int_\varepsilon^{x}\frac{u(y)-u(x)-u'(x)(y-x)}{(x-y)^{\alpha+2}}dy.
\end{align*} 

Consequently, for all $x\in(2\varepsilon,l)$ we obtain
\begin{align*}
(D^\alpha_x u)_x(x)=\frac{1}{\Gamma(1-\alpha)}&\left(\frac{\alpha(u(0)-u(x))+(\alpha+1)u'(x)x}{x^{\alpha+1}}\right.\\
&\quad\left.+\alpha(\alpha+1)\int_0^{x}\frac{u(y)-u(x)-u'(x)(y-x)}{(x-y)^{\alpha+2}}dy\right).
\end{align*}
After changing the variable of integration by setting $z=\hat{x}-y$, we reach \eqref{p:another} since $\varepsilon$ is arbitrary.
\end{proof}

The following lemma states the maximum principle, which is valid due to \eqref{e:another}.

\begin{lemma}\label{l:maximum}
Assume that $u\in C^2(0,l)\cap C[0,l)$ with $u'\in L^1(0,l)$ attains its maximum at $\hat{x}\in(0,l)$.
Then $(D_x^\alpha u)_x(\hat{x})\le0$.
\end{lemma}

\begin{proof}
Since $u'(\hat x)=0$, then the right-hand side of \eqref{e:another} is non-positive and the claim follows.
\end{proof}

For the sake of convenience, we introduce the following operators, when $u:(0,l)\to \mathbb{R}$ is a measurable function and $p\in\mathbb{R}$,
\begin{align*}
&J[u,p](x)=\frac{\alpha(u(0)-u(x))+(\alpha+1)px}{x^{\alpha+1}\Gamma(1-\alpha)},\\
&K_{(a,b)}[u,p](x)=\frac{\alpha(\alpha+1)}{\Gamma(1-\alpha)}\int_a^b[u(x-z)-u(x)+pz]\frac{dz}{z^{\alpha+2}},
\end{align*}
where $0\le a<b<x$.
We write also $J^\alpha[u,p](x)$ and $K^\alpha(a,b)[u,p](x)$ when making clear the dependence of $\alpha$. 
For a function $u:[0,l)\times A\to\mathbb{R}$, where $A\subset\mathbb{R}$ is an interval, we write $J[u, p](x,y):=J[u(\cdot,y), p](x)$ and $K_{(a,b)}[u,p](x,y):=K_{(a,b)}[u(\cdot,y), p](x)$ for $y\in A$.

\begin{proposition}\label{p:propertyofK}
Let $u$ be a real-valued upper semicontinuous function on $[0,l)$.  If $x\in(0,l)$, $\delta\in(0,x)$, and $p\in\mathbb{R}$, then $K_{(\delta,x)}[u,p](x)$ and it is bounded from above.
If $u\in C^2(0,l)\cap C[0,l)$ and $p_x=u'(x)$, then $K_{(0,x)}[u,p_x](x)$ is well-defined for each $x\in(0,l)$.
\end{proposition}

\begin{remark}
Is is clear from the definition that for any real-valued lower semicontinuous functions $u$ on $[0,l)$, the operator $K_{(\delta,x)}[u,p](x)$ is well-defined for all $x\in (0,l)$ and $\delta\in(0,x)$. Moreover, it is bounded from below.
\end{remark}

\begin{proof}[Proof of Proposition \ref{p:propertyofK}] 
We only deal with the case that $u$ is upper semicontinuous, because the last assertion is obvious from Proposition \ref{p:another}.

Since $z\mapsto u(x-z)-u(x)+pz$ is upper semicontinuous on $[x-a,x]$ for $x\in(0,l)$, it attains its maximum and 
$$
(u(x-z)-u(x)+pz)^+\le \left[\max_{[a,x]}(u(x-z)-u(x)+pz)\right]\vee 0\quad\text{for all $z\in[x-\delta,x]$.}
$$
Here, $a \vee b:=\max\{a,b\}$ for $a,b\in\mathbb{R}$ and we will also use $a\wedge b:=\min\{a,b\}$ throughout this paper. 
Since the right-hand side is integrable on $(x-a,x)$, we see that $K_{(x-a,x)}[u,p](x)$ is well-defined.
It is straightforward to see that $K_{(x-a,x)}[u,p](x)$ is bounded from above.
\end{proof}

\begin{lemma}\label{l:important}
Let $\alpha\in(0,1)$, $a\in(0,l)$, $\varphi\in C^2(a,l)\cap C[a,l)$ and $\hat{x}\in(a,l)$. Let $u$ be a real-valued upper semicontinuous function on $[0,l)$.
Also, for each $\varepsilon>0$ let $\alpha_\varepsilon\in(0,1)$ and $x_\varepsilon\in(a,l)$ be sequences and let $u_\varepsilon$ be a sequence of real-valued upper semicontinuous functions on $[0,l)$. Assume that $u_\varepsilon\le u$ on $[0,l)$ and
\begin{equation}\label{e:lemmaconv}
\lim_{\varepsilon\to0}(\alpha_\varepsilon,x_\varepsilon,u_\varepsilon(x_\varepsilon))=(\alpha,\hat{x},u(\hat{x})).
\end{equation}
Set $p_\varepsilon=\varphi'(x_\varepsilon)$ and $p=\varphi'(\hat{x})$.
Then,
\begin{itemize}
\item[(i)] $\limsup_{\varepsilon\to0}J^{\alpha_\varepsilon}[u_\varepsilon,p_\varepsilon](x_\varepsilon)\le J^\alpha[u,p](\hat{x})$,
\item[(ii)] $\lim_{\varepsilon\to0}K_{(0,x_\varepsilon-a)}^{\alpha_\varepsilon}[\varphi,p_\varepsilon](x_\varepsilon)=K_{(0,\hat{x}-a)}^\alpha[\varphi,p](\hat{x})$,
\item[(iii)] $\limsup_{\varepsilon\to0}K_{(x_\varepsilon-a,x_\varepsilon)}^{\alpha_\varepsilon}[u_\varepsilon,p_\varepsilon](x_\varepsilon)\le K_{(\hat{x}-a,\hat{x})}^\alpha[u,p](\hat{x})$.
\end{itemize}
Moreover, if $\varphi\in C^2(0,l)\cap C[0,l)$, then $x\mapsto K_{(0,x)}[\varphi,p_x](x)$ with $p_x=\varphi'(x)$ is continuous in $(0,l)$.
\end{lemma}

\begin{proof}
(i) Using the assumption $u_\varepsilon(0)\le u(0)$, we see that
\begin{equation*}
\begin{split}
\limsup_{\varepsilon\to0}J^{\alpha_\varepsilon}[u_\varepsilon,p_\varepsilon](x_\varepsilon)
&\le\limsup_{\varepsilon\to0}\frac{\alpha_\varepsilon(u(0)-u_\varepsilon(x_\varepsilon))+(\alpha_\varepsilon+1)p_\varepsilon x_\varepsilon}{x_\varepsilon^{\alpha_\varepsilon+1}\Gamma(1-\alpha_\varepsilon)}\\
&= J^\alpha[u,p](\hat{x}).
\end{split}
\end{equation*}

(ii) Let $a'\in(0,\hat{x})$ be such that 
$[\hat{x}-2a',\hat{x}+2a']\subset(a,l)$.
We may assume that $x_\varepsilon\in(\hat{x}-a',\hat{x}+a')$ by letting $\varepsilon$ smaller if necessary.
Since $\varphi\in C^2(a,l)$, then for $z\in (0, x_\varepsilon-a)$  we have,
$$
\varphi(x_\varepsilon-z)-\varphi(x_\varepsilon)+\varphi'(x_\varepsilon)z=
z^2\int_0^1\int_0^1 \varphi''(x_\varepsilon -st z)t\,ds dt.
$$
Hence,
\begin{equation}\label{rzalfa}
 K^{\alpha_\varepsilon}_{(0,x_\varepsilon -a)}[\varphi, p_\varepsilon](x_\varepsilon) =
\int_0^{x_\varepsilon-a} \frac 1{z^{\alpha_\varepsilon}}
\int_0^1\int_0^1 \varphi''(x_\varepsilon -st z)t\,ds dt.
\end{equation}
We fix any $\delta>0$, then we find $\eta>0$ such that
\begin{equation}\label{reta}
 \int_{\hat x - a - \eta}^{\hat x - a} \frac{\max\{ \varphi''(\zeta):\ \zeta \in [\hat x - a - \eta,\hat x - a] \}}{z^{\max \alpha_\varepsilon}} < \delta.
\end{equation}
Thus,
\begin{align*}
 & K^{\alpha_\varepsilon}_{(0,x_\varepsilon -a)}[\varphi, p_\varphi](x_\varepsilon) -
 K^{\alpha}_{(0, \hat x-a)}[\varphi, p](\hat x) \\
 =& K^{\alpha_\varepsilon}_{(\hat x - a - \eta,x_\varepsilon -a)}[\varphi, p_\varphi](x_\varepsilon) -
 K^{\alpha}_{(\hat x - a - \eta, \hat x-a)}[\varphi, p](\hat x) \\ &+
 K^{\alpha_\varepsilon}_{(0,\hat x - a - \eta}[\varphi, p_\varphi](x_\varepsilon) -
 K^{\alpha}_{(0, \hat x-a - \eta)}[\varphi, p](\hat x).
\end{align*}
By (\ref{reta}) the first difference above is estimated by  $2\delta$. The second difference is less than $\delta$ for sufficiently small $\varepsilon$ due to (\ref{rzalfa}) and the Lebesgue dominated convergence Theorem. 
Hence, our claim follows.

(iii) By our assumptions, we immediately see that 
\begin{equation}\label{e:lemmabound}
\sup_{0<\varepsilon\ll1}(|u_\varepsilon(x_\varepsilon)|\vee |p_\varepsilon|)\le C
\end{equation}
for a constant $C>0$.
Since $u_\varepsilon\le u$ in $[0,l)$ and $u$ is upper semicontinuous, then
$$
u_\varepsilon(x_\varepsilon-z)\le u(x_\varepsilon-z)\le \max_{[0,a]}u\quad\text{for all $z\in[x_\varepsilon-a,x_\varepsilon]$.}
$$
Thus, we have the following estimate,
\begin{align*}
&(u_\varepsilon(x_\varepsilon-z)-{u}_\varepsilon(x_\varepsilon)-p_\varepsilon z)\mathds{1}_{(x_\varepsilon-a,x_\varepsilon)}(z)z^{-\alpha_\varepsilon-2}\\
\le& (\max_{[0,a]}u+{\max_{[a,l]}u}+Cz)\mathds{1}_{(x_\varepsilon-a,x_\varepsilon)}(z)z^{-\alpha_\varepsilon-2}\\
\le& \left|\max_{[0,a]}u+{\max_{[a,l]}u}+Cz\right|\mathds{1}_{(x_\varepsilon-a,x_\varepsilon)}(z)(z^{-2}\vee z^{-1})\\
\le& \left|\max_{[0,a]}u+{\max_{[a,l]}u}+Cz\right|\mathds{1}_{(a'',l)}(z)(z^{-2}\vee z^{-1})\quad\text{for all $z\in(\delta,l)$,}
\end{align*}
where $\delta = \frac12(\hat{x} -a)$.
Since the right-hand side is integrable on $(\delta ,l)$, Fatou's lemma yields the desired inequality.
\end{proof}

\begin{remark}
(i) A symmetric statement which ``upper semicontinuous" is replaced with ``lower semicontinuous" is true since $J[u,p](x)=-J[-u,-p](x)$.

(ii) In the following sections, we use Lemma \ref{l:important} for functions that also depend on the time variable. We can not use it directly because it is stated for a single variable function. We may state the corresponding result as follows:\\
Let us suppose that $u_\varepsilon, u:[0,l) \times \Lambda \to \mathbb{R}$ are measurable, where $\Lambda\subset \mathbb{R}$ is an interval. If for all $t\in \Lambda$ functions $u_\varepsilon(\cdot,t)$, $ u(\cdot,t)$ satisfy the assumptions of Lemma \ref{l:important}, then the claim holds for $u_\varepsilon(\cdot,t)$, $ u(\cdot,t)$ and all $t\in \Lambda$.
\end{remark}

\section{Definition of a solution}
In this section we propose a notion of a solution of the initial boundary value problem
\begin{equation}\label{e:master}
u_t=(D^\alpha_x u)_x+f\quad\text{in $Q_T$}
\end{equation}
and
\begin{equation}\label{e:initialboundary}
u=g\quad\text{on $\partial_pQ_T$.}
\end{equation}
Here, $Q_T=(0,l)\times(0,T)$ and $\partial_pQ_T$ stands for the parabolic boundary, i.e., $\partial_pQ_T=([0,l]\times\{0\})\cup(\{0,l\}\times[0,T))$. We also introduce $Q_{T,0}:= (0,l)\times [0,T)$.
Throughout this paper, the given functions $f:Q_T\to\mathbb{R}$ and $g:\partial_pQ_T\to\mathbb{R}$ are assumed to be continuous.

To motivate the definition of solutions that we call viscosity solutions, we suppose that $u\in\mathcal{C}(Q_{T,0})$ and $u-\varphi$ attains a maximum over $Q_{T,0}$ at $(\hat{x},\hat{t})\in Q_T$ for a function $\varphi\in\mathcal{C}(Q_{T,0})$.
Here, $\mathcal{C}(Q_{T,0})$ is a space of test functions which we set as
\begin{equation}\label{dfC}
 \mathcal{C}(Q_{T,0})=\{\varphi\in C^{2,1}(Q_T)\cap C(Q_{T,0})\mid \text{$\varphi_x(\cdot,t)\in L^1(0,{l})$ for every $t\in(0,T)$}\}.
\end{equation}
The classical maximum principle and Lemma \ref{l:maximum} yield $u_t=\varphi_t$ and $(D_x^\alpha u)_x\ge(D_x^\alpha\varphi)_x$ at $(\hat{x},\hat{t})$.
Thus, if $u$ satisfies \eqref{e:master} pointwise in $Q_T$, then
$$
\varphi_t(\hat{x},\hat{t})\le(D_x^\alpha\varphi)_x(\hat{x},\hat{t})+f(\hat{x},\hat{t}).
$$
Since this inequality does not include the derivative of $u$, we are tempted to use it to define a generalized subsolution for $u$ which is not differentiable. The opposite inequality comes out if one replaces the maximum with a minimum. 

Let $\Omega$ be a set in $\mathbb{R}^2$.
For a function $w:\Omega\to\mathbb{R}$ let $w^*$ and $w_*$ denote the upper semicontinuous envelope and the lower semicontinuous envelope, respectively.
Namely, 
$$
w^*(z)=\lim_{\delta\to0^+}\sup\{w(\zeta)\mid \zeta\in \Omega\cap\overline{B_\delta(z)}\}
$$
and $w_*=-(-w)^*$.
Here and hereafter, $B_\delta(z)$ is an open ball in $\mathbb{R}^2$ centered at $z$ with radius $\delta$, i.e., $B_\delta(z)=\{\zeta\in\mathbb{R}^d\mid |z-\zeta|<\delta\}$, and $\overline{B_\delta(z)}$ is its closure.

\begin{definition}[Viscosity solution]\label{d:visc}
We say that a real-valued function $u$ on $Q_{T,0}$ is a \emph{viscosity subsolution} (resp. \emph{viscosity supersolution}) of \eqref{e:master} if $u^*<\infty$ (resp. $u_*>-\infty$) in $Q_{T,0}$ and, for every $((\hat{x},\hat{t}),\varphi)\in Q_T\times \mathcal{C}(Q_{T,0})$ that satisfies $\max_{Q_{T,0}}(u^*-\varphi)=(u^*-\varphi)(\hat{x},\hat{t})$ (resp. $\min_{Q_{T,0}}(u_*-\varphi)=(u_*-\varphi)(\hat{x},\hat{t})$),
$$
\varphi_t(\hat{x},\hat{t})\le (D_x^\alpha\varphi)_x(\hat{x},\hat{t})+f(\hat{x},\hat{t})\quad(\text{resp. $\varphi_t(\hat{x},\hat{t})\ge (D_x^\alpha\varphi)_x(\hat{x},\hat{t})+f(\hat{x},\hat{t})$}).
$$
Moreover, we say that a real-valued function $u$ on $Q_T\cup\partial_pQ_T$ is a viscosity subsolution (resp. viscosity supersolution) of \eqref{e:master}-\eqref{e:initialboundary} if $u^*<\infty$ (resp. $u_*>-\infty$) in $Q_T\cup\partial_pQ_T$, $u$ is a viscosity subsolution (resp. viscosity supersolution) of \eqref{e:master}, and satisfies $u^*\le g$ (resp. $u_*\ge g$) on $\partial_pQ_T$.

If a real-valued function $u$ on $Q_{T,0}$ (resp. $Q_T\cup\partial_pQ_T$) is a viscosity sub- and supersolution of \eqref{e:master} (resp. \eqref{e:master}-\eqref{e:initialboundary}), we say that $u$ is a \emph{viscosity solution} of \eqref{e:master} (resp. \eqref{e:master}-\eqref{e:initialboundary}). 
\end{definition}

The notion of viscosity solution by Definition \ref{d:visc} is consistent with that of ``classical solution'' that satisfies \eqref{e:master} pointwise in $Q_T$.

\begin{proposition}[Consistency]\label{p:consistency}
Let $u\in\mathcal{C}(Q_{T,0})$, (see (\ref{dfC}) for the definition of this set).
Then, $u$ is a viscosity solution of \eqref{e:master} if and only if $u$ satisfies \eqref{e:master} pointwise in $Q_T$.
\end{proposition}

\begin{proof}
We saw the `if' part before Definition \ref{d:visc}.
The `only if' part is straightforward since $u$ can be taken as a test function.
\end{proof}

After establishing the consistency result, we suppress the word ``viscosity'' from now on. 

There are several equivalent definitions of solutions. We utilize these definitions to establish the existence and uniqueness of solutions and some propeties.

\begin{proposition}[Alternative definitions]\label{p:alter}
Let $u$ be a real-valued function on $Q_{T,0}$ with $u^*<+\infty$ in $Q_{T,0}$. Then, the following statements are equivalent:

\begin{itemize}
\item[(i)] $u$ is a subsolution of \eqref{e:master};
\item[(ii)] for every $((\hat{x},\hat{t}),\varphi)\in Q_T\times(C^{2,1}(Q_T)\cap C(Q_{T,0}))$ that satisfies $\max_{Q_{T,0}}(u^*-\varphi)=(u^*-\varphi)(\hat{x},\hat{t})$,
$$
\varphi_t(\hat{x},\hat{t})\le J[\varphi,p](\hat{x},\hat{t})+K_{(0,\hat{x})}[\varphi,p](\hat{x},\hat{t})+f(\hat{x},\hat{t})
$$
holds with $p=\varphi_x(\hat{x},\hat{t})$; 
\item[(iii)] for every $((\hat{x},\hat{t}),\varphi)\in Q_T\times C^{2,1}(Q_T)$ that satisfies $\max_{Q_T}(u^*-\varphi)=(u^*-\varphi)(\hat{x},\hat{t})$, 
\begin{equation*}
\varphi_t(\hat{x},\hat{t}) 
\le J[u^*,p](\hat{x},\hat{t})+K_{(0,\hat{x}-\delta)}[\varphi,p](\hat{x},\hat{t}) 
+K_{(\hat{x}-\delta,\hat{x})}[u^*,p](\hat{x},\hat{t})+f(\hat{x},\hat{t})
\end{equation*}
holds for all $\delta\in(0,\hat{x})$ with $p=\varphi_x(\hat{x},\hat{t})$;
\item[(iv)] for every $((\hat{x},\hat{t}),\varphi)\in Q_T\times C^{2,1}(Q_T)$ that satisfies $\max_{Q_T}(u^*-\varphi)=(u^*-\varphi)(\hat{x},\hat{t})$, $K_{(0,\hat{x})}[u^*,p](\hat{x},\hat{t})$ with $p=\varphi_x(\hat{x},\hat{t})$ is well-defined and
$$
\varphi_t(\hat{x},\hat{t})\le J[u^*,p](\hat{x},\hat{t})+K_{(0,\hat{x})}[u^*,p](\hat{x},\hat{t})+f(\hat{x},\hat{t})
$$
holds.
\end{itemize}
\end{proposition}

\begin{proof}
The proofs of implications (ii) $\Rightarrow$ (i) and (iv) $\Rightarrow$ (ii) are easy.
In fact, the former is a direct consequence of Proposition \ref{p:another}. To prove the latter, let $((\hat{x},\hat{t}),\varphi)\in Q_T\times (C^{2,1}(Q_T)\cap C(Q_{T,0}))$ be such that $\max_{Q_{T,0}}(u^*-\varphi)=(u^*-\varphi)(\hat{x},\hat{t})$.
Since $K_{(0,\hat{x})}[u^*,p](\hat{x},\hat{t})$ with $p=\varphi_x(\hat{x},\hat{t})$ exists by (iv) and $u^*(\hat{x}-z,\hat{t})-u^*(\hat{x},\hat{t})+pz\le \varphi(\hat{x}-z,\hat{t})-\varphi(\hat{x},\hat{t})+pz$ holds for all $z\in[0,\hat{x}]$, we have
$$
J[u^*,p](\hat{x},\hat{t})+K_{(0,\hat{x})}[u^*,p](\hat{x},\hat{t})\le J[\varphi,p](\hat{x},\hat{t})+K_{(0,\hat{x})}[\varphi,p](\hat{x},\hat{t}).
$$
The desired inequality is immediately obtained from the inequality by (iv).

We shall prove the implication (i) $\Rightarrow$ (iii).
Let $((\hat{x},\hat{t}),\varphi)\in Q_T\times C^{2,1}(Q_T)$ be such that $\max_{Q_T}(u^*-\varphi)=(u^*-\varphi)(\hat{x},\hat{t})$ and fix $\delta\in(0,\hat{x})$ arbitrarily. 
We set $\psi:=\varphi+(u^*-\varphi)(\hat{x},\hat{t})$ so that $\max_{Q_T}(u^*-\psi)=(u^*-\psi)(\hat{x},\hat{t})=0$.

Since $u^*$ is upper semicontinuous in $Q_{T,0}$, there exists a sequence $u_\varepsilon\in C(Q_{T,0})$ such that $u_\varepsilon\searrow u^*$ pointwise in $Q_{T,0}$ as $\varepsilon\to0^+$.
Also, there exists a sequence $\psi_\varepsilon\in\mathcal{C}(Q_{T,0})$ such that $\psi_\varepsilon=\psi$ in $\overline{B_{(\hat{x}-\delta)/2}(\hat{x},\hat{t})}\cap Q_{T,0}$, $u^*\le\psi_\varepsilon\le\psi$ in $B_{\hat{x}-\delta}(\hat{x},\hat{t})\cap Q_{T,0}$, and $u^*\le\psi_\varepsilon\le u_\varepsilon+\varepsilon$ in $Q_{T,0}\setminus B_{\hat{x}-\delta}(\hat{x},\hat{t})$.
Observe that (a) $\psi_\varepsilon=\psi=u^*$ at $(\hat{x},\hat{t})$, (b) $\lim_{\varepsilon\to0}\psi_\varepsilon=u^*$ pointwise in $Q_{T,0}\setminus B_{\hat{x}-\delta}(\hat{x},\hat{t})$, and (c) $\psi_\varepsilon\le\psi=\varphi$ in $B_{\hat{x}-\delta}(\hat{x},\hat{t})$.

It is clear that $\max_{Q_{T,0}}(u^*-\psi_\varepsilon)=(u^*-\psi_\varepsilon)(\hat{x},\hat{t})$.
Thus, by (i) we have
\begin{equation}\label{e:alter1}
(\psi_\varepsilon)_t(\hat{x},\hat{t})\le (D_x^\alpha\psi_\varepsilon)_x(\hat{x},\hat{t})+f(\hat{x},\hat{t}).
\end{equation}
Since $\psi_\varepsilon=\psi$ near $(\hat{x},\hat{t})$, we see that $(\psi_\varepsilon)_t=\psi_t=\varphi_t$ and $(\psi_\varepsilon)_x=\psi_x=\varphi_x$ at $(\hat{x},\hat{t})$.
Proposition \ref{p:another} yields $(D_x^\alpha\psi_\varepsilon)_x(\hat{x},\hat{t})=J[\psi_\varepsilon,p_\varepsilon](\hat{x},\hat{t})+K_{(0,\hat{x})}[\psi_\varepsilon,p_\varepsilon](\hat{x},\hat{t})$ with $p_\varepsilon=(\psi_\varepsilon)_x(\hat{x},\hat{t})$.
Accordingly, \eqref{e:alter1} can be rewritten as
\begin{equation}\label{e:alter2}
\varphi_t(\hat{x},\hat{t})\le J[\psi_\varepsilon,p](\hat{x},\hat{t})+K_{(0,\hat{x})}[\psi_\varepsilon,p](\hat{x},\hat{t})+f(\hat{x},\hat{t}),
\end{equation}
where $p=\psi_x(\hat{x},\hat{t})$.

It is easy to check from (a) and (b) that $\lim_{\varepsilon\to0}J[\psi_\varepsilon,p](\hat{x},\hat{t})=J[u^*,p](\hat{x},\hat{t})$.
Since (a) and (c) imply that 
\begin{align*}
\psi_\varepsilon(\hat{x}-z,\hat{t})-\psi_\varepsilon(\hat{x},\hat{t})+pz
&\le\psi(\hat{x}-z,\hat{t})-\psi(\hat{x},\hat{t})+pz\\
&=\varphi(\hat{x}-z,\hat{t})-\varphi(\hat{x},\hat{t})+pz
\end{align*}
holds for all $z\in(0,\delta)$, then we have $K_{(0,\delta)}[\psi_\varepsilon,p]\le K_{(0,\delta)}[\varphi,p](\hat{x},\hat{t})$. The estimate $\limsup_{\varepsilon\to0}K_{(\delta,\hat{x})}[\psi_\varepsilon,p](\hat{x},\hat{t})\le K_{(\delta,\hat{x})}[u^*,p](\hat{x},\hat{t})$ immediately follows from Lemma \ref{l:important}.
Keeping in mind these estimates, while taking the limit supremum as $\varepsilon\to0$ in \eqref{e:alter2} yields the desired inequality.

We finish the proof of this proposition by showing the implication (iii) $\Rightarrow$ (iv).
Let $((\hat{x},\hat{t}),\varphi)\in Q_T\times C^{2,1}(Q_T)$ be such that $\max_{Q_T}(u^*-\varphi)=(u^*-\varphi)(\hat{x},\hat{t})$.
We define
\begin{align*}
v_\delta(z):=(\varphi(\hat{x}-z,\hat{t})-\varphi(\hat{x},\hat{t})+pz)\mathds{1}_{[0,\hat{x}-\delta]}(z)+(u^*(\hat{x}-z,\hat{t})-u^*(\hat{x},\hat{t})+pz)\mathds{1}_{(\hat{x}-\delta,\hat{x}]}(z)
\end{align*}
for $\delta\in(0,\hat{x})$ and set
$$
K[v_\delta]:=\frac{\alpha(\alpha+1)}{\Gamma(1-\alpha)}\int_0^{\hat{x}}v_\delta(z)\frac{dz}{z^{\alpha+2}}=K_{(0,\hat{x}-\delta)}[\varphi,p](\hat{x},\hat{t})+K_{(\hat{x}-\delta,\hat{x})}[u,p](\hat{x},\hat{t}).
$$
Since $u^*(\hat{x}-z,\hat{t})-u^*(\hat{x},\hat{t})+pz\le \varphi(\hat{x}-z,\hat{t})-\varphi(\hat{x},\hat{t})+pz$ for all $z\in[0,\hat{x})$, we know that $v_\delta^+\searrow v_0^+$ and $v_\delta^-\nearrow v_0^-$ as $\delta\to0$.

By (iii) the inequality
\begin{equation*}
\begin{split}
\varphi_t(\hat{x},\hat{t})&\le J[u^*,p](\hat{x},\hat{t})+K_{(0,\hat{x}-\delta)}[\varphi,p](\hat{x},\hat{t})\\
&\quad+K_{(\hat{x}-\delta,\hat{x})}[u^*,p](\hat{x},\hat{t})+f(\hat{x},\hat{t})
\end{split}
\end{equation*}
holds for all $\delta\in(0,\hat{x})$ with $p=\varphi_x(\hat{x},\hat{t})$.
From this inequality we see
\begin{align*}
0\le K[v_\delta^+]-K[v_\delta^-]+C\le K[v_{\delta'}^+]-K[v_\delta^-]+C\le K[v_{\delta'}^+]+C<+\infty
\end{align*}
for all $\delta<\delta'$, where $C=J[u^*,p](\hat{x},\hat{t})-\varphi_t(\hat{x},\hat{t})$ and $\delta'$ is a fixed constant with $0<\delta'<\hat{x}$. Therefore it turns out that $\lim_{\delta\to0}K[v_\delta^\pm]=K[v_0^\pm]$ by the monotone convergence theorem and $K[v_0^\pm]$ is finite. Thus we get the desired result.
\end{proof}

For $(x,t)\in Q_T$ we denote by $\mathcal{N}_{x,t}$ a family of neighborhoods $N$ of $(x,t)$ in $Q_T$ such that every $N$ includes the line segment between $(x,t)$ and $(y,t)$ whenever $(y,t)\in N$ and $0<y<x$.
Evidently, $Q_T\in\mathcal{N}_{x,t}$ for all $(x,t)\in Q_T$.

\begin{proposition}\label{p:localstrict}
Let $u$ be a real-valued function on $Q_{T,0}$ with $u^*<+\infty$ in $Q_{T,0}$.
Then $u$ is a subsolution of \eqref{e:master} if and only if, for every $((\hat{x},\hat{t}),\varphi)\in Q_T\times C^{2,1}(Q_T)$ and $N\in\mathcal{N}_{\hat{x},\hat{t}}$ such that $u^*-\varphi$ attains a strict maximum on $N$ at $(\hat{x},\hat{t})$ in the sense that $(u^*-\varphi)(x,t)<(u^*-\varphi)(\hat{x},\hat{t})$ for all $(x,t)\in N\setminus\{(\hat{x},\hat{t})\}$,
\begin{align*}
\varphi_t(\hat{x},\hat{t})&\le J[u^*,p](\hat{x},\hat{t})+K_{(0,\hat{x}-\delta)}[\varphi,p](\hat{x},\hat{t})\\
&\quad+K_{(\hat{x}-\delta,\hat{x})}[u^*,p](\hat{x},\hat{t})+f(\hat{x},\hat{t})
\end{align*}
holds for all $\delta\in(0,\hat{x})$ that satisfies $(\delta,\hat{t})\in N$ with $p=\varphi_x(\hat{x},\hat{t})$.
\end{proposition}

\begin{proof}
We use Proposition \ref{p:alter} (iii) for proofs of both implications.
We first prove the `if' part.
Let $((\hat{x},\hat{t}),\varphi)\in Q_T\times C^{2,1}(Q_T)$ be such that $\max_{Q_T}(u^*-\varphi)=(u^*-\varphi)(\hat{x},\hat{t})$.
Set $\psi_\varepsilon(x,t)=\varphi(x,t)+\varepsilon(|x-\hat{x}|^2+|t-\hat{t}|^2)$ for a small parameter $\varepsilon>0$.
Then $u^*-\psi_\varepsilon$ attains a strict maximum on $Q_T$ at $(\hat{x},\hat{t})$ so
\begin{equation}\label{e:localstrict}
\begin{split}
(\psi_\varepsilon)_t(\hat{x},\hat{t})&\le J[u^*,p_\varepsilon](\hat{x},\hat{t})+K_{(0,\hat{x}-\delta)}[\psi_\varepsilon,p_\varepsilon](\hat{x},\hat{t})\\
&\quad+K_{(\hat{x}-\delta,\hat{x})}[u^*,p_\varepsilon](\hat{x},\hat{t})+f(\hat{x},\hat{t})
\end{split}
\end{equation}
holds for all $\delta\in(0,\hat{x})$ with $p_\varepsilon=(\psi_\varepsilon)_x(\hat{x},\hat{t})$.
It is straightforward to see that $(\psi_\varepsilon)_t(\hat{x},\hat{t})=\varphi_t(\hat{x},\hat{t})$, $p_\varepsilon=\varphi_x(\hat{x},\hat{t})=:p$, and
$$
K_{(0,\hat{x}-\delta)}[\psi_\varepsilon,p_\varepsilon]=K_{(0,\hat{x}-\delta)}[\varphi,p](\hat{x},\hat{t})+\frac{\alpha(1+\alpha)\hat{x}^{1-\alpha}}{(1-\alpha)\Gamma(1-\alpha)}\varepsilon.
$$ 
Therefore sending $\varepsilon\to0$ in \eqref{e:localstrict} yields the desired inequality.

Let $((\hat{x},\hat{t}),\varphi)\in Q_T\times C^{2,1}(Q_T)$ and $N\in\mathcal{N}_{\hat{x},\hat{t}}$ be such that $u^*-\varphi$ attains a strict maximum on $N$ at $(\hat{x},\hat{t})$.
We denote by $\bar{\varphi}$ an extension of $\varphi$ to $Q_T$ such that $\bar{\varphi}\in C^{2,1}(Q_T)$ and $\max_{Q_T}(u^*-\bar{\varphi})=(u^*-\bar{\varphi})(\hat{x},\hat{t})$.
Noticing that $\bar{\varphi}_t=\varphi_t$ and $\bar{\varphi}_x=\varphi_x$ at $(\hat{x},\hat{t})$, we have
\begin{align*}
\varphi_t(\hat{x},\hat{t})&\le J[u^*,p](\hat{x},\hat{t})+K_{(0,\hat{x}-\delta)}[\bar{\varphi},p](\hat{x},\hat{t})\\
&\quad+K_{(\hat{x}-\delta,\hat{x})}[u^*,p](\hat{x},\hat{t})+f(\hat{x},\hat{t})
\end{align*}
holds for all $\delta\in(0,\hat{x})$ with $p=\varphi_x(\hat{x},\hat{t})$.
Since $\delta$ may be taken so that $(\delta,\hat{t})\in N$ and then $K_{(0,\hat{x}-\delta)}[\bar{\varphi},p](\hat{x},\hat{t})=K_{(0,\hat{x}-\delta)}[\varphi,p](\hat{x},\hat{t})$, this inequality is nothing but the desired one.
\end{proof}

\begin{remark}
(i) Symmetric statements in Propositions \ref{p:alter} and \ref{p:localstrict} hold for a supersolution of \eqref{e:master}.

(ii) Also in Proposition \ref{p:alter} (ii) and (iv) the maximum may be replaced by a strict maximum and in (iv) it may be replaced by a local strict maximum in the sense that, for a neighborhood $N$ of $(\hat{x},\hat{t})$ in $Q_T$,
$$
(u^*-\varphi)(x,t)<(u^*-\varphi)(\hat{x},\hat{t})\quad\text{for all $(x,t)\in N$.}
$$ 
However, in (ii) the locality may not be allowed.
\end{remark}

\begin{proposition}\label{p:terminal}
Assume that $f$ is continuous.
Let $u$ be a subsolution (resp. supersolution) of \eqref{e:master} in $Q_T$.
Then $u$ is a subsolution (resp. supersolution) of \eqref{e:master} in $Q_T^*:=(0,l)\times(0,T]$ provided that $u^*(x,T)<+\infty$ (resp. $u_*(x,T)>-\infty$) for all $x\in(0,l)$.
\end{proposition}

\begin{proof}
We only prove for subsolutions.
It suffices to show that
\begin{equation}\label{e:terminal}
\varphi_t(\hat{x},T)\le J[u^*,p](\hat{x},T)+K_{(0,\hat{x}-\delta)}[\varphi,p](\hat{x},T)+K_{(\hat{x}-\delta,\hat{x})}[u^*,p](\hat{x},T)+f(\hat{x},T)
\end{equation}
holds for all $\delta\in(0,\hat{x})$ with $p=\varphi_x(\hat{x},\hat{t})$ whenever $u^*-\varphi$ attains a strict maximum on $Q_T^*$ at $(\hat{x},T)$ with $0<\hat{x}<l$ for $\varphi\in  C^{2,1}(Q_{T,0}^*)$, where $Q_{T,0}^*=[0,l)\times(0,T]$.
Fix $\delta\in(0,\hat{x})$ arbitrarily.

For $\varepsilon>0$ we define $\varphi_\varepsilon(x,t):=\varphi(x,t)+\varepsilon/(T-t)$.
It is a standard fact (see, e.g., the proof of \cite[Theorem 3.2.10]{Giga}) that there is a maximum point $(x_\varepsilon,t_\varepsilon)\in Q_T$ of $u^*-\varphi_\varepsilon$ on $Q_T^*$ and
$$
\lim_{\varepsilon\to0}(x_\varepsilon,t_\varepsilon,u^*(x_\varepsilon,t_\varepsilon))=(\hat{x},\hat{t},u^*(\hat{x},\hat{t})).
$$
We may assume that $\delta<\inf_\varepsilon x_\varepsilon$ by restricting to smaller $\varepsilon$.
Since $u$ is a subsolution of \eqref{e:master} in $Q_T$, we have
\begin{align*}
(\varphi_\varepsilon)_t(x_\varepsilon,t_\varepsilon)\le &J[u^*,p_\varepsilon](x_\varepsilon,t_\varepsilon)+K_{(0,x_\varepsilon-\delta)}[\varphi_\varepsilon,p_\varepsilon](x_\varepsilon,t_\varepsilon)\\
&+K_{(x_\varepsilon-\delta,x_\varepsilon)}[u^*,p_\varepsilon](x_\varepsilon,t_\varepsilon)+f(x_\varepsilon,t_\varepsilon),
\end{align*}
where $p_\varepsilon=(\varphi_\varepsilon)_x(x_\varepsilon,t_\varepsilon)$.
Since $(\varphi_\varepsilon)_t(x_\varepsilon,t_\varepsilon)\ge \varphi_t(x_\varepsilon,t_\varepsilon)$ and $p_\varepsilon=\varphi_x(x_\varepsilon,t_\varepsilon)$, sending $\varepsilon\to0$ using Lemma \ref{l:important} yields \eqref{e:terminal}.
\end{proof}

\begin{remark}\label{r:terminal}
Also for each statements in Proposition \ref{p:alter}, $Q_T$ and $Q_{T,0}$ may be replaced by $Q_T^*$ and $Q_{T,0}^*$, respectively.
\end{remark}

\begin{remark}
We cannot offer examples of explicit solution satisfying the homogeneous Dirichlet boundary condition 
\begin{equation}\label{e:initialboundary2}
u(0,t)=u(l,t)=0\quad\text{for $t\in[0,T)$.}
\end{equation}
However,  in a recent study, \cite{Rys}, the author showed that the operator $(D^\alpha_x)_x$ defined on $D((D^\alpha_x)_x)\subset L^2(0,l)$, whose elements satisfy the mixed boundary data,
$$
u_x(0)=u(l)=0
$$
generates an analytic semigroup.
\end{remark}

\section{Comparison principle}
We shall establish uniqueness of solutions via the comparison principle.

\begin{theorem}\label{t:comparison}
Let $u$ be a subsolution while $v$ be a supersolution of \eqref{e:master}.
If $-\infty<u^*\le v_*<+\infty$ on $\partial_p Q_T$, then $u^*\le v_*$ in $Q_T\cup\partial_p Q_T$.
\end{theorem}

\begin{proof}
We fix any $T'$ smaller than $T$.
It is enough to prove that $u^*\le v_*$ in $\overline{Q_{T'}}$, when $u^*\le v_*$ on $\partial_pQ_{T'}^*$.
This can be shown by using the conventional doubling variable technique.

We denote $T'$ by $T$ again. We assume that $u$ is a subsolution (respectively, $v$ is a supersolution) of \eqref{e:master} in $Q_T^*$,  and $-\infty<u^*,v_*<+\infty$ in $\overline{Q_T}$. 
Let us suppose the contrary, i.e., $\theta:=\max_{\overline{Q_T}}(u^*-v_*)>0$.

We define a real-valued function $\Phi$ on $\overline{Q_T}\times \overline{Q_T}$ by
$$
\Phi(x,t,y,s):=u^*(x,t)-v_*(y,s)-\frac{|x-y|^2+|t-s|^2}{\varepsilon},
$$
where $\varepsilon>0$ is a parameter.
Since the function $\Phi$ is upper semicontinuous on the compact set $\overline{Q_T}\times \overline{Q_T}$, it attains a maximum at some point $(x_\varepsilon,t_\varepsilon,y_\varepsilon,s_\varepsilon)\in\overline{Q_T}\times \overline{Q_T}$.
We claim that this points are in $Q_T^*\times Q_T^*$ for sufficiently small $\varepsilon$.
To see this, we shall check that, possibly after extracting a subsequence which is not relabled,
\begin{equation}\label{e:comp1}
\lim_{\varepsilon\to0}(x_\varepsilon,t_\varepsilon,y_\varepsilon,s_\varepsilon,u^*(x_\varepsilon,t_\varepsilon),v_*(y_\varepsilon,s_\varepsilon))=(\hat{x},\hat{t},\hat{x},\hat{t},u^*(\hat{x},\hat{t}),v_*(\hat{x},\hat{t})),
\end{equation}
where $(\hat{x},\hat{t})$ is a point such that $(u^*-v_*)(\hat{x},\hat{t})=\theta$.
The convergence of the point sequence is established according to \cite[Lemma 3.1]{CrandallIshiiLions}.
We can select another subsequence (not relabled) such that $\lim_{\varepsilon\to0} v_*(y_\varepsilon,s_\varepsilon)= \liminf_{\varepsilon\to0}v_*(y_\varepsilon,s_\varepsilon)$.
Then, due to the lower semicontinuity of $v_*$, we have
\begin{eqnarray*}
(u^* -v_*)(\hat{x},\hat{t})& = &
\lim_{\varepsilon\to0}(u^*(x_\varepsilon,t_\varepsilon) - v_*(y_\varepsilon,s_\varepsilon)) =
\lim_{\varepsilon\to0}u^*(x_\varepsilon,t_\varepsilon) - \lim_{\varepsilon\to0} v_*(y_\varepsilon,s_\varepsilon)\\
&\le &\lim_{\varepsilon\to0}u^*(x_\varepsilon,t_\varepsilon) - v_*(\hat{x},\hat{t}).
\end{eqnarray*}
This implies that  $\liminf_{\varepsilon\to0}u^*(x_\varepsilon,t_\varepsilon)\ge u^*(\hat{x},\hat{t})$, hence 
$$
\lim_{\epsilon\to 0} u^*(x_\varepsilon,t_\varepsilon) =u^*(\hat{x},\hat{t}) 
$$
follows.
By the same method, the convergence of $v_*(y_\varepsilon,s_\varepsilon)$ is also established.
Now, we see that $(\hat{x},\hat{t})\not\in\partial_pQ_T^*$; otherwise, $(u^*-v_*)(\hat{x},\hat{t})\le0$ by the assumption but this is contradictory since $\theta>0$.
Therefore our claim is proved.

We also claim that  
\begin{equation}\label{e:comp1.5}
\lim_{\varepsilon\to0}\frac{|x_\varepsilon-y_\varepsilon|^2}{\varepsilon}=0.
\end{equation}
Since $(x_\varepsilon,t_\varepsilon,y_\varepsilon,s_\varepsilon)$ is a maximum point of $\Phi$, we have $\Phi(x_\varepsilon,t_\varepsilon,y_\varepsilon,s_\varepsilon)\ge \Phi(\hat{x},\hat{t},\hat{x},\hat{t})$, that is,
\begin{equation*}
\frac{|x_\varepsilon-y_\varepsilon|^2+|t_\varepsilon-s_\varepsilon|^2}{\varepsilon}\le u^*(x_\varepsilon,t_\varepsilon)-v_*(y_\varepsilon,s_\varepsilon)-u^*(\hat{x},\hat{t})+v_*(\hat{x},\hat{t}).
\end{equation*}
The right-hand side vanishes as $\varepsilon\to0$ due to \eqref{e:comp1}, so as the left-hand side.

Since $(x,t)\mapsto\Phi(x,t,y_\varepsilon,s_\varepsilon)$ attains a maximum on $Q_T^*$ at $(x_\varepsilon,t_\varepsilon)$, by Proposition \ref{p:alter} (iv) and Proposition \ref{p:terminal} (see also Remark \ref{r:terminal}), we see that $K_{(0,x_\varepsilon)}[u^*,p_\varepsilon](x_\varepsilon,t_\varepsilon)$ with $p_\varepsilon=2(x_\varepsilon-y_\varepsilon)/\varepsilon$ is well-defined and
\begin{equation}\label{e:comp3}
\frac{2(t_\varepsilon-s_\varepsilon)}{\varepsilon}\le J[u^*,p_\varepsilon](x_\varepsilon,t_\varepsilon)+K_{(0,x_\varepsilon)}[u^*,p_\varepsilon](x_\varepsilon,t_\varepsilon)+f(x_\varepsilon,t_\varepsilon).
\end{equation}
Similarly, since $(y,s)\mapsto -\Phi(x_\varepsilon,t_\varepsilon,y,s)$ attains a minimum on $Q_T^*$ at $(y_\varepsilon,s_\varepsilon)$, then $K_{(0,y_\varepsilon)}[v_*,p_\varepsilon](y_\varepsilon,s_\varepsilon)$ is well-defined and
\begin{equation}\label{e:comp4}
\frac{2(t_\varepsilon-s_\varepsilon)}{\varepsilon}\ge J[v_*,p_\varepsilon](y_\varepsilon,s_\varepsilon)+K_{(0,y_\varepsilon)}[v_*,p_\varepsilon](y_\varepsilon,s_\varepsilon)+f(y_\varepsilon,s_\varepsilon).
\end{equation}
Subtracting \eqref{e:comp4} from \eqref{e:comp3} yields
\begin{equation}\label{e:comp5}
\begin{split}
0
&\le J[u^*,p_\varepsilon](x_\varepsilon,t_\varepsilon)-J[v_*,p_\varepsilon](y_\varepsilon,s_\varepsilon)\\
&\quad+K_{(0,x_\varepsilon)}[u^*,p_\varepsilon](x_\varepsilon,t_\varepsilon)-K_{(0,y_\varepsilon)}[v_*,p_\varepsilon](y_\varepsilon,s_\varepsilon)\\
&\quad+f(x_\varepsilon,t_\varepsilon)-f(y_\varepsilon,s_\varepsilon).
\end{split}
\end{equation}
We are going to take a limit as $\varepsilon\to0$ in this inequality in order to obtain a contradiction.
For this purpose we estimate 
$$
J_\varepsilon:=J[u^*,p_\varepsilon](x_\varepsilon,t_\varepsilon)-J[v_*,p_\varepsilon](y_\varepsilon,s_\varepsilon)
$$
and
$$
K_\varepsilon:=K_{(0,x_\varepsilon)}[u^*,p_\varepsilon](x_\varepsilon,t_\varepsilon)-K_{(0,y_\varepsilon)}[v_*,p_\varepsilon](y_\varepsilon,s_\varepsilon)
$$
as $\varepsilon\to0$.
We note that it is not possible to apply Lemma \ref{l:important} because $p_\varepsilon$ need not be bounded.
In what follows let $\delta>0$ be a constant with $\delta<\inf_\varepsilon(x_\varepsilon\wedge y_\varepsilon)$. 

Since $u^*$ and $-v_*$ are upper semicontinuous and we have \eqref{e:comp1}, then
\begin{align*}
&\limsup_{\varepsilon\to0}\left(\frac{u^*(0,t_\varepsilon)-u^*(x_\varepsilon,t_\varepsilon)}{x_\varepsilon^{\alpha+1}}-\frac{v_*(0,s_\varepsilon)-v_*(y_\varepsilon,s_\varepsilon)}{y_\varepsilon^{\alpha+1}}\right)\\
&\le \frac{(u^*-v_*)(0,\hat{t})-(u^*-v_*)(\hat{x},\hat{t})}{\hat{x}^{\alpha+1}}\le\frac{-\theta}{\hat{x}^{\alpha+1}}.
\end{align*}
The last inequality is due to the boundary conditions and the definition of $\theta$.
By the inequality
$$
\left|\frac{1}{x^\alpha}-\frac{1}{y^\alpha}\right|\le\frac{\alpha}{\delta^{\alpha+1}}|x-y|\quad\text{for all $x,y\in(\delta,\infty)$,}
$$
and \eqref{e:comp1.5}, we also see that
$$
\lim_{\varepsilon\to0}\left(\frac{p_\varepsilon}{x_\varepsilon^\alpha}-\frac{p_\varepsilon}{y_\varepsilon^\alpha}\right)
\le \frac{\alpha}{\delta^{\alpha+1}}\lim_{\varepsilon\to0}|p_\varepsilon||x_\varepsilon-y_\varepsilon|=\frac{\alpha}{\delta^{\alpha+1}}\lim_{\varepsilon\to0}\frac{|x_\varepsilon-y_\varepsilon|^2}{\varepsilon}=0.
$$
Thus these estimates give $\limsup_{\varepsilon\to0}J_\varepsilon\le-\theta/\hat{x}^{\alpha+1}$.

We estimate $K_\varepsilon$ as $\varepsilon\to0$ from 
$$
K_{\varepsilon,1}:=K_{(0,\delta)}[u^*,p_\varepsilon](x_\varepsilon,t_\varepsilon)-K_{(0,\delta)}[v_*,p_\varepsilon](y_\varepsilon,s_\varepsilon)
$$
and
$$
K_{\varepsilon,2}:=K_{(\delta,x_\varepsilon)}[u^*,p_\varepsilon](x_\varepsilon,t_\varepsilon)-K_{(\delta,y_\varepsilon)}[v_*,p_\varepsilon](y_\varepsilon,s_\varepsilon).
$$
We know that $\Phi(x_\varepsilon,t_\varepsilon,y_\varepsilon,s_\varepsilon)\ge\Phi(x_\varepsilon-z,t_\varepsilon,y_\varepsilon-z,s_\varepsilon)$ for all $z\in[0,\delta]$, because $(x_\varepsilon,t_\varepsilon,y_\varepsilon,s_\varepsilon)$ is a maximum point of $\Phi$.
From this observation, the estimate $K_{\varepsilon,1}\le0$ follows immediately.
By the upper semicontinuity of $u^*$ and \eqref{e:comp1} we see that $u^*\le \max_{\overline{Q_T}}u^*$ in $Q_T^*$ and $u^*(x_\varepsilon,t_\varepsilon)\ge u^*(\hat{x},\hat{t})-C$,  for sufficiently small $\varepsilon$, where $C$ is a positive constant. Thus,
\begin{align*}
[u^*(x_\varepsilon-z,t_\varepsilon)-u^*(x_\varepsilon,t_\varepsilon)]\mathds{1}_{(\delta,x_\varepsilon)}(z)
&\le \left[\max_{\overline{Q_T}}u^*-u^*(\hat{x},\hat{t})+C\right]\mathds{1}_{(\delta,x_\varepsilon)}(z)\\
&\le\left|\max_{\overline{Q_T}}u^*-u^*(\hat{x},\hat{t})+C\right|\mathds{1}_{(\delta,l)}(z)
\end{align*}
for all $z\in(0,l)$ such that $x_\varepsilon >z $.
The right-hand side multiplied by $z^{-\alpha-2}$ is integrable on $(0,l)$ so the Lebesgue dominated convergence theorem implies
\begin{equation}\label{e:comp6}
\limsup_{\varepsilon\to0}\int_\delta^{x_\varepsilon}[u^*(x_\varepsilon-z,t_\varepsilon)-u^*(x_\varepsilon,t_\varepsilon)]\frac{dz}{z^{\alpha+2}}\le \int_\delta^{\hat{x}}[u^*(\hat{x}-z,\hat{t})-u^*(\hat{x},\hat{t})]\frac{dz}{z^{\alpha+2}}.
\end{equation}
By a symmetric argument, by Fatou's lemma we also have
\begin{equation}\label{e:comp7}
\liminf_{\varepsilon\to0}\int_\delta^{y_\varepsilon}[v_*(y_\varepsilon-z,s_\varepsilon)-v_*(y_\varepsilon,s_\varepsilon)]\frac{dz}{z^{\alpha+2}}\ge \int_\delta^{\hat{x}}[v_*(\hat{x}-z,\hat{t})-v_*(\hat{x},\hat{t})]\frac{dz}{z^{\alpha+2}}.
\end{equation}
It would not be difficult to see that
\begin{align*}
\limsup_{\varepsilon\to0}\int_{y_\varepsilon}^{x_\varepsilon}\frac{p_\varepsilon}{z^{\alpha+1}}dz
\le \limsup_{\varepsilon\to0}\frac{p_\varepsilon}{(x_\varepsilon\wedge y_\varepsilon)^{\alpha+1}}\int_{y_\varepsilon}^{x_\varepsilon}dz
=\limsup_{\varepsilon\to0}\frac{(x_\varepsilon-y_\varepsilon)^2}{\varepsilon(x_\varepsilon\wedge y_\varepsilon)^{\alpha+1}}=0.
\end{align*}
Combining this with \eqref{e:comp6} and \eqref{e:comp7} we find that
\begin{align*}
\limsup_{\varepsilon\to0}K_{\varepsilon,2}
&=\limsup_{\varepsilon\to0}C_\alpha\left(\int_\delta^{x_\varepsilon}[u^*(x_\varepsilon-z,t_\varepsilon)-u^*(x_\varepsilon,t_\varepsilon)]\frac{dz}{z^{\alpha+2}}\right.\\
&\quad\left.-\int_\delta^{y_\varepsilon}[v_*(y_\varepsilon-z,s_\varepsilon)-v_*(y_\varepsilon,s_\varepsilon)]\frac{dz}{z^{\alpha+2}}+\int_{y_\varepsilon}^{x_\varepsilon}\frac{p_\varepsilon}{z^{\alpha+1}}dz\right)\\
&\le C_\alpha\int_\delta^{\hat{x}}[(u^*-v_*)(\hat{x}-z,\hat{t})-(u^*-v_*)(\hat{x},\hat{t})]\frac{dz}{z^{\alpha+2}}\le0,
\end{align*}
where $C_\alpha=\alpha(\alpha+1)/\Gamma(1-\alpha)$. 
Therefore we get the estimate $\limsup_{\varepsilon\to0}K_\varepsilon\le0$.

Taking the limit supremum in \eqref{e:comp5} as $\varepsilon\to0$ yields
$$
0\le-\frac{\theta}{\hat{x}^{\alpha+1}},
$$
a contradiction since $\theta>0$.
\end{proof}

The uniqueness of solutions is the direct consequence of Theorem \ref{t:comparison}.
\begin{corollary}[Uniqueness of solutions]
Let $u$ and $v$ be solutions of \eqref{e:master}-\eqref{e:initialboundary}. Then $u\equiv v$ in $Q_T\cup\partial_pQ_T$.
\end{corollary}

Finally, in this section we show the weak maximum principle as a simple application of Theorem \ref{t:comparison}.

\begin{corollary}[Weak maximum principle]\label{c:contraction}
Let $f_1,f_2:Q_T\to\mathbb{R}$ be continuous functions.
Let $u$ be a subsolution of 
\begin{equation}\label{e:comparison1}
u_t=(D_x^\alpha u)_x+f_1(x,t)\quad\text{in $Q_T$}
\end{equation}
and let $v$ be a supersolution of
\begin{equation*}
v_t=(D_x^\alpha v)_x+f_2(x,t)\quad\text{in $Q_T$.}
\end{equation*}
Assume that $-\infty<u^*,v_*<+\infty$ on $\partial_pQ_T$.
Then, 
\begin{equation}\label{e:contraction}
\sup_{Q_T\cup\partial_pQ_T}(u^*-v_*)^+(x,t)\le \sup_{\partial_pQ_T}(u^*-v_*)^+ +\int_0^T\sup_{x\in (0,l)}\,
(f_1-f_2)^+(x,s)ds.
\end{equation}
In particular,
\begin{equation}\label{e:comparison2}
\sup_{Q_T\cup\partial_pQ_T}(u^*)^+=\sup_{\partial_pQ_T}(u^*)^+\quad\text{if $f_1\le0$}
\end{equation}
and
\begin{equation}\label{e:comparison3}
\inf_{Q_T\cup\partial_pQ_T}(v_*)^-=\inf_{\partial_pQ_T}(v_*)^-\quad\text{if $f_2\ge0$.}
\end{equation}
\end{corollary}

\begin{proof}
We may assume that
$$
C:=\sup_{\partial_pQ_T}(u^*-v_*)^++\int_0^T\sup_{x\in(0,l)}(f_1-f_2)^+(x,s)ds<+\infty
$$
for all $t\in[0,T)$; otherwise \eqref{e:contraction} is automatically established.
It is not difficult to check that 
$$
\tilde{v}(x,t):=v_*(x,t)+\sup_{\partial_pQ_T}(u^*-v_*)^++\int_0^t\sup_{x\in(0,l)}(f_1-f_2)^+(x,s)ds
$$ 
is a supersolution of \eqref{e:comparison1} and $u^*\le \tilde{v}$ on $\partial_pQ_T$.
Thus, by Theorem \ref{t:comparison} we have $u^*\le \tilde{v}$ in $Q_T\cup\partial_pQ_T$ and obtain \eqref{e:contraction}.

If we put $f_2\equiv0$ and $v=0$ in \eqref{e:contraction}, then
$$
\sup_{Q_T\cup\partial_pQ_T}(u^*)^+\le \sup_{\partial_pQ_T}(u^*)^+ + \int_0^T\sup_{(0,l)}f_1^+(\cdot,s)ds.
$$ 
Thus, if $f_1\le0$, we have $\sup_{Q_T\cup\partial_pQ_T}(u^*)^+\le\sup_{\partial_pQ_T}(u^*)^+$. The converse inequality is always true, hence we get \eqref{e:comparison2}. We also get \eqref{e:comparison3} by arguing similarly.
\end{proof}

\section{Existence of solutions}
We shall construct a (continuous) solution of the initial boundary value problem \eqref{e:master}-\eqref{e:initialboundary} by Perron's method (see \cite{Ishii}) under a certain condition of the initial boundary data $g$.
First, in Subsection \ref{ss:existence}, we show the existence of (possibly discontinuous) solutions under the hypothesis that there exist suitable subsolutions and supersolutions of \eqref{e:master}-\eqref{e:initialboundary}. Specifically, we give a construction of a subsolution in Lemma \ref{l:stabsup}, and through Lemma \ref{l:super} we show that it is in fact also a supersolution in Lemma \ref{l:perron} and hence a solution. In Subsection \ref{ss:construction}, we construct suitable subsolutions and supersolutions and guarantee the existence of the solution; Theorem \ref{t:existence}. As its by-product, we obtain the fact that the solution is bounded and continuous. Its  uniqueness follows from the comparison theorem.

\subsection{Existence by Perron's method}\label{ss:existence}

\begin{lemma}\label{l:stabsup}
Assume that $f$ is continuous.
Let $S_-$ and $S_+$ be nonempty sets of subsolutions and supersolutions of \eqref{e:master}, respectively.
Let  functions $u$ and $v$ be defined by
$$
u(x,t)=\sup\{w(x,t) \mid w\in S_-\}\quad\text{and}\quad v(x,t)=\inf\{w(x,t)\mid w\in S_+\}
$$
for $(x,t)\in Q_{T,0}$.
Then, $u^*$ (resp. $v_*$) is a subsolution (resp.  supersolution) of \eqref{e:master} provided that $u^*<+\infty$ (resp. $v_*>-\infty$) in $Q_{T,0}$, respectively.
\end{lemma}

\begin{proof}
We perform the proof only for $u$, a subsolution, since the argument of a supersolution $v$ is the same. Let $((\hat{x},\hat{t}),\varphi)\in Q_T\times C^{2,1}(Q_T)$ be such that $u^*-\varphi$ attains a strict maximum on $Q_T$ at $(\hat{x},\hat{t})$.
We fix $\delta\in(0,\hat{x})$ arbitrarily. The goal is to show that
\begin{equation}\label{e:stabsup1}
\varphi_t(\hat{x},\hat{t})\le J[u^*,p](\hat{x},\hat{t})+K_{(0,\hat{x}-\delta)}[\varphi,p](\hat{x},\hat{t})+K_{(\hat{x}-\delta,\hat{x})}[u^*,p](\hat{x},\hat{t})+f(\hat{x},\hat{t}),
\end{equation}
where $p=\varphi_x(\hat{x},\hat{t})$.

According to \cite[Lemma V.1.6]{BardiCapuzzoDolcetta}, there exist a sequence $((x_\varepsilon,t_\varepsilon),u_\varepsilon)\in Q_T\times S_-$ and a neighborhood with compact closure $N\in \mathcal{N}_{\hat{x},\hat{t}}$ with $(\delta,\hat{t})\in N$ such that $(x_\varepsilon,t_\varepsilon)$ is a maximum point for $u_\varepsilon^*-\varphi$ on $N$ and
\begin{equation*}\label{e:stabconv}
\lim_{\varepsilon\to0}(x_\varepsilon,t_\varepsilon,u_\varepsilon^*(x_\varepsilon,t_\varepsilon))=(\hat{x},\hat{t},u^*(\hat{x},\hat{t})).
\end{equation*}
(For the proof, refer to the argument leading to \eqref{e:comp1}.)
We may assume that $\delta<\inf_\varepsilon x_\varepsilon$ by restricting to smaller $\varepsilon$.
Since $u_\varepsilon$ is a subsolution of \eqref{e:master}, by Proposition \ref{p:localstrict} we have
\begin{equation}\label{e:subj}
\begin{split}
\varphi_t(x_\varepsilon,t_\varepsilon)\le& J[u_\varepsilon^*,p_\varepsilon](x_\varepsilon,t_\varepsilon)+K_{(0,x_\varepsilon-\delta)}[\varphi,p_\varepsilon](x_\varepsilon,t_\varepsilon)\\
&+K_{(x_\varepsilon-\delta,x_\varepsilon)}[u_\varepsilon^*,p_\varepsilon](x_\varepsilon,t_\varepsilon)+f(x_\varepsilon,t_\varepsilon),
\end{split}
\end{equation}
where $p_\varepsilon=\varphi_x(x_\varepsilon,t_\varepsilon)$.
The definition of $u$ implies $u_\varepsilon^*\le u^*$ in $Q_{T,0}$ and hence Lemma \ref{l:important} is now applicable. It yields,
\begin{align*}
&\limsup_{\varepsilon\to0}(J[u_\varepsilon^*,p_\varepsilon](x_\varepsilon,t_\varepsilon)+K_{(0,x_\varepsilon-\delta)}[\varphi,p_\varepsilon](x_\varepsilon,t_\varepsilon)+K_{(x_\varepsilon-\delta,x_\varepsilon)}[u_\varepsilon^*,p_\varepsilon](x_\varepsilon,t_\varepsilon))\\
&\le J[u^*,p](\hat{x},\hat{t})+K_{(0,\hat{x}-\delta)}[\varphi,p](\hat{x},\hat{t})+K_{(\hat{x}-\delta,\hat{x})}[u^*,p](\hat{x},\hat{t}).
\end{align*}
Therefore, we get \eqref{e:stabsup1} by taking the limit supremum in \eqref{e:subj} as $\varepsilon\to0$.
\end{proof}

\begin{lemma}\label{l:super}
Let $\eta$ be a supersolution of \eqref{e:master} and let $S_-$ be a nonempty set of subsolutions $v$ of \eqref{e:master} that satisfies $v\le \eta$ in $Q_{T,0}$.
If $u_*\in S_-$ is not a supersolution of \eqref{e:master} while $u_*>-\infty$ in $Q_{T,0}$, then there exist a function $w$ such that $w\in S_-$ and a point $(y,s)\in Q_T$ such that $u(y,s)<w(y,s)$.
\end{lemma}

\begin{proof}
Since $u_*\in S_-$ is not a supersolution of \eqref{e:master}, there is $((\hat{x},\hat{t}),\varphi)\in Q_T\times (C^{2,1}(Q_T)\cap C(Q_{T,0}))$ such that $u_*-\varphi$ attains a strict minimum on $Q_{T,0}$ at $(\hat{x},\hat{t})$ and
\begin{equation}\label{e:existence1}
\varphi_t(\hat{x},\hat{t})<J[\varphi,p](\hat{x},\hat{t})+K_{(0,\hat{x})}[\varphi,p](\hat{x},\hat{t})+f(\hat{x},\hat{t}),
\end{equation}
where $p=\varphi_x(\hat{x},\hat{t})$; see Proposition \ref{p:alter} (ii).
We may assume that $(u_*-\varphi)(\hat{x},\hat{t})=0$ by replacing $\varphi$ with $\varphi+(u_*-\varphi)(\hat{x},\hat{t})$ if necessary.
It follows immediately from Lemma \ref{l:important} that $(x,t)\mapsto J[\varphi,p_{x,t}](x,t)+K_{(0,x)}[\varphi,p_{x,t}](x,t)$ with $p_{x,t}=\varphi_x(x,t)$ is continuous in $Q_T$. Thus there is $r>0$ such that  
\begin{equation}\label{e:existence2}
\varphi_t(x,t)<J[\varphi,p_{x,t}](x,t)+K_{(0,x)}[\varphi,p_{x,t}](x,t)+f(x,t)
\end{equation}
holds for all $(x,t)\in B_{2r}:=Q_T\cap B_{2r}(\hat{x},\hat{t})$.

We see that $u^*$ satisfies $\varphi\le u_*\le \eta$ in $Q_{T,0}$ by the definition.
Suppose $\varphi=\eta_*$ at $(\hat{x},\hat{t})$.
Then $\min_{Q_{T,0}}(\eta_*-\varphi)=(\eta_*-\varphi)(\hat{t},\hat{x})$ so \eqref{e:existence1} is contradictory since $\eta$ is a supersolution of \eqref{e:master}.
Thus we know that $\varphi<\eta$ at $(\hat{x},\hat{t})$.
We set $\lambda:=\frac{1}{2}(\eta-\varphi)(\hat{x},\hat{t})>0$.
The lower semicontinuity of $\eta_*-\varphi$ implies that $\varphi+\lambda\le \eta$ in $B_{2r}$ by letting $r$ smaller if necessary.
Since $u_*>\varphi$ in $Q_{T,0}\setminus\{(\hat{x},\hat{t})\}$, there is $\lambda'\in(0,\lambda)$ such that $\varphi+2\lambda'\le u_*$ in $Q_{T,0}\setminus B_r$.

Define $w:Q_{T,0}\to\mathbb{R}$ by
\begin{equation*}
  w(x,t)=\begin{cases}
    u(x,t)\vee (\varphi(x,t)+\lambda')\quad&\text{for $(x,t)\in B_r$,}\\
    u(x,t)\quad&\text{for $(x,t)\in Q_{T,0}\setminus B_r$.}
  \end{cases}
\end{equation*} 
We show that $w$ is the desirable function in the statement of this lemma.
In order to show that $w\in S_-$, it suffices to prove that $w^*$ is a subsolution of \eqref{e:master} since it is clear that $w\le \eta$ in $Q_{T,0}$ by the construction.
To this end, we take$((\hat{y},\hat{s}),\psi)\in Q_T\times(C^{1,2}(Q_T)\cap C(Q_{T,0}))$ such that $\max_{Q_{T,0}}(w^*-\psi)=(w^*-\psi)(\hat{y},\hat{s})$ and aim to show
\begin{equation}\label{e:existence2.5}
\psi_t(\hat{y},\hat{s})\le J[\psi,q](\hat{y},\hat{s})+K_{(0,\hat{y})}[\psi,q](\hat{y},\hat{s})+f(\hat{y},\hat{s}).
\end{equation}
We may assume that $(w^*-\psi)(\hat{y},\hat{s})=0$.

In the case that $w^*=u^*$ at $(\hat{y},\hat{s})$, we see $\max_{Q_{T,0}}(u^*-\psi)=(u^*-\psi)(\hat{y},\hat{s})$.
Since $u^*$ is a subsolution of \eqref{e:master}, then \eqref{e:existence2.5} is obtained by Lemma \ref{l:stabsup}.

In the case that $w^*=\varphi+\lambda'$ at $(\hat{y},\hat{s})$, we see that $\max_{Q_{T,0}}(\varphi+\lambda'-\psi)=(\varphi+\lambda'-\psi)(\hat{y},\hat{s})=0$.
Evidently, $\psi_y=\varphi_x$ and $\psi_s=\varphi_t$ at $(\hat{y},\hat{s})$.
By definition of $w$, it is clear that $w^*(\hat{y}-z,\hat{s})\ge\varphi(\hat{y}-z,\hat{s})+\lambda'$ for all $z\in[0,\hat{y}]$ and hence
\begin{align*}
\psi(\hat{y}-z,\hat{s})-\psi(\hat{y},\hat{s})-\varphi(\hat{y}-z,\hat{s})+\varphi(\hat{y},\hat{s})
&=\psi(\hat{y}-z,\hat{s})-\varphi(\hat{y}-z,\hat{s})-\lambda'\\
&\ge\psi(\hat{y}-z,\hat{s})-w^*(\hat{y}-z,\hat{s})\\
&\ge0.
\end{align*}
This implies that 
$$J[\psi,q](\hat{y},\hat{s})\ge J[\varphi,p_{\hat{y},\hat{s}}](\hat{y},\hat{s})\quad\hbox{and}\quad K_{(0,\hat{y})}[\psi,q](\hat{y},\hat{s})\ge K_{(0,\hat{y})}[\varphi,p_{\hat{y},\hat{s}}](\hat{y},\hat{s}),$$ 
where $q=\psi_y(\hat{y},\hat{s})$.
Since $w=u$ in $Q_{T,0}\setminus B_r$, $(\hat{y},\hat{s})\in B_r$ and thus, by using \eqref{e:existence2}, we obtain \eqref{e:existence2.5} as
\begin{align*}
&\psi_t(\hat{y},\hat{s})-J[\psi,q](\hat{y},\hat{s})-K_{(0,\hat{y})}[\psi,q](\hat{y},\hat{s})-f(\hat{y},\hat{s})\\
&\le \varphi_t(\hat{y},\hat{s})-J[\varphi,p_{\hat{y},\hat{s}}](\hat{y},\hat{s})-K_{(0,\hat{y})}[\varphi,p_{\hat{y},\hat{s}}](\hat{y},\hat{s})-f(\hat{y},\hat{s})\\
&\le0.
\end{align*}

There is a sequence $(x_\varepsilon,t_\varepsilon)\in Q_{T,0}$ such that $\lim_{\varepsilon\to0}(x_\varepsilon,t_\varepsilon,u(x_\varepsilon,t_\varepsilon))=(\hat{x},\hat{t},u_*(\hat{x},\hat{t}))$. Then we have
$$
\liminf_{\varepsilon\to0}(w(x_\varepsilon,t_\varepsilon)-u(x_\varepsilon,t_\varepsilon))\ge \lim_{\varepsilon\to0}(\varphi(x_\varepsilon,t_\varepsilon)+\lambda'-u(x_\varepsilon,t_\varepsilon))=\lambda'>0.
$$
This means that there is a point $(x,t)\in Q_T$ such that $w(x,t)>u(x,t)$.
The proof is now complete.
\end{proof}

\begin{lemma}\label{l:perron}
Assume that $f$ and $g$ are continuous. Let $\xi$ be a subsolution (respectively, $\eta$ be a supersolution) of \eqref{e:master}-\eqref{e:initialboundary},  satisfying $\eta^*<+\infty$ and $\xi_*>-\infty$ in $Q_T\cup\partial_pQ_T$.
Suppose that $\xi\le\eta$ in $Q_T\cup\partial_pQ_T$ and $\xi_*=\eta^*=g$ in $\partial_pQ_T$.
Then, there exists a (possibly discontinuous) solution $u$ of \eqref{e:master}-\eqref{e:initialboundary} that satisfies $\xi\le u\le \eta$ in $Q_T\cup\partial_pQ_T$ and $u=g$ on $\partial_pQ_T$.
\end{lemma}

\begin{proof}
Let $S$ be a set of subsolutions $w$ of \eqref{e:master}-\eqref{e:initialboundary} that satisfies $w\le \eta$ in $Q_T\cup\partial_pQ_T$.
We notice that $S\neq\emptyset$ since $\xi\in S$.
We define
$$
u(x,t):=\sup\{w(x,t)\mid w\in S\}\quad\text{for $(x,t)\in Q_T\cup\partial_pQ_T$}.
$$
Lemma \ref{l:stabsup} ensures that $u^*$ is a subsolution of \eqref{e:master}.
If $u_*$ were not a supersolution of \eqref{e:master}, then by Lemma \ref{l:super} there would exist a subsolution $w$ of \eqref{e:master} and a point $(y,s)\in Q_T$ such that $u(y,s)<w(y,s)$.
But this contradicts the maximality of $u$. Thus, $u_*$ must be a supersolution of \eqref{e:master}.
Thus we see that $u$ is a solution of \eqref{e:master}.
It is clear from the definition of $u$ that $\xi\le u\le \eta$ in $Q_T\cup\partial_pQ_T$, so that $\xi_*\le u_*\le u\le u^*\le\eta^*$ in $Q_T\cup\partial_pQ_T$.
Since $\xi_*=\eta^*=g$ on $\partial_pQ_T$, we have $u^*=u_*=u=g$ on $\partial_pQ_T$.
\end{proof}

\subsection{Construction of suitable sub- and supersolutions}\label{ss:construction}

In order to obtain a subsolution and a supersolution satisfying the condition of Lemma \ref{l:perron}, we construct subsolutions and supersolutions agreeing with the boundary data and the initial data in Proposition \ref{p:lateral} and Proposition \ref{p:bottom}, respectively. Here, we only present the proof for subsolutions since the same is applied to supersolutions. The method of construction follows the conventional one, e.g., \cite{CrandallKocanLionsSwiech} (see also \cite{BirindelliDemengel}).

\begin{proposition}\label{p:lateral}
Assume that $f$ is bounded continuous and $g$ is uniformly continuous.
Then, there are a bounded subsolution $\xi_1$ and a bounded supersolution $\eta_1$ of \eqref{e:master}-\eqref{e:initialboundary} that satisfy $\xi_1=\eta_1=g$ on $\{0,l\}\times[0,T)$.
\end{proposition}

\begin{proof}
For $(y,s)\in\{0,l\}\times[0,T)$ and $\varepsilon>0$ we define $\xi_1^{y,s,\varepsilon}:Q_T\cup\partial_p Q_T\to\mathbb{R}$ by
$$
\xi_1^{y,s,\varepsilon}(x,t)=g(y,s)-2\varepsilon-(M_1+M_2+\|f\|_\infty)\rho^{{y}}(x)-M_2|t-s|.
$$ 
Here $M_1$ and $M_2$ are positive constants to be chosen later and
\begin{equation}\label{e:rho}
	\rho^{{y}}(x)=\begin{cases}
		\displaystyle-\frac{1}{\Gamma(2+\alpha)}x^{1+\alpha}+\frac{C}{\Gamma(1+\alpha)}x^\alpha\quad&\text{when $y=0$,}\\
		\displaystyle-\frac{1}{\Gamma(2+\alpha)}(x^{1+\alpha}-L^{1+\alpha})\quad&\text{when $y=l$,}
	\end{cases}
\end{equation}
where $C>l/(1+\alpha)$. Subsequently, we will supress the superindex $y$.
Note that $\rho\in C^2(0,l)\cup C[0,l)$ and it satisfies $\rho'\in L^1(0,l)$, $\rho(y)=0$, $\rho>0$ in $[0,l]\setminus\{y\}$, and $(D_x^\alpha \rho)_x=-1$ in $(0,l)$.
The last one can be verified using the well-known formula
\begin{equation}\label{e:wellknown}
D_x^\alpha x^\beta=\frac{\Gamma(\beta+1)}{\Gamma(\beta-\alpha+1)}x^{\beta-\alpha}\quad\text{for $\beta>-1$.}
\end{equation}

We claim that $\xi_1^{y,s,\varepsilon}$ is a subsolution of \eqref{e:master}-\eqref{e:initialboundary} if $M_1$ and $M_2$ are taken large enough.
To see this, we first take $((\hat{x},\hat{t}),\varphi)\in Q_T\times C^{2,1}(Q_T)$ such that $\max_{Q_T}(\xi_1^{y,s,\varepsilon}-\varphi)=(\xi_1^{y,s,\varepsilon}-\varphi)(\hat{x},\hat{t})$.
Then, $\varphi_t(\hat{x},\hat{t})\le M_2$ because
\begin{align*}
\varphi(\hat{x},\hat{t})-\varphi(\hat{x},\hat{t}-h)
&\le\xi_1^{y,s,\varepsilon}(\hat{x},\hat{t})-\xi_1^{y,s,\varepsilon}(\hat{x},\hat{t}-h)\\
&= -M_2(|\hat{t}-s|-|\hat{t}-h-s|)\\
&\le M_2h
\end{align*}
for small $h>0$.
Moreover, by setting $p:=\varphi_x(\hat{x},\hat{t})=(\xi_1^{y,s,\varepsilon})_x(\hat{x},\hat{t})$, Proposition \ref{p:another} and the definition of $\rho$ imply
\begin{align*}
&J[\xi_1^{y,s,\varepsilon},p](\hat{x},\hat{t})+K_{(0,\hat{x})}[\xi_1^{y,s,\varepsilon},p](\hat{x},\hat{t})\\
&=-(M_1+M_2+\|f\|_\infty)(J[\rho,\rho'(\hat{x})](\hat{x},\hat{t})+K_{(0,\hat{x})}[\rho,\rho'(\hat{x})](\hat{x},\hat{t}))\\
&=-(M_1+M_2+\|f\|_\infty)(D_x^\alpha\rho)_x(\hat{x})\\
&=M_1+M_2+\|f\|_\infty.
\end{align*}
Thus, we have
\begin{align*}
&\varphi_t(\hat{x},\hat{t})-J[\xi_1^{y,s,\varepsilon},p](\hat{x},\hat{t})-K_{(0,\hat{x})}[\xi_1^{y,s,\varepsilon},p](\hat{x},\hat{t})-f(\hat{x},\hat{t})\\
&\le M_2-(M_1+M_2+\|f\|_\infty)-f(\hat{x},\hat{t})\le 0,
\end{align*}
which means that $\xi_1^{y,s,\varepsilon}$ is a subsolution of \eqref{e:master}.

We next choose $M_1$ and $M_2$ so that
$$
M_1\ge \sup_{x\in[0,l]\setminus\{y\}}\frac{(\omega(|x-y|)-\varepsilon)^+}{\rho(x)},\quad M_2\ge \sup_{t\in[0,T)\setminus\{s\}}\frac{(\omega(|t-s|)-\varepsilon)^+}{|t-s|},
$$
where $\omega:[0,\infty)\to[0,\infty)$ is a continuity modulus of $g$ on $\partial_p Q_T$:
$$
|g(x,t)-g(y,s)|\le \omega(|x-y|+|t-s|)\quad\text{for all $(x,t),(y,s)\in\partial_p Q_T$.}
$$
Then, for all $(x,t)\in\partial_pQ_T$,
\begin{equation}\label{e:bound}
\begin{split}
\xi_1^{y,s,\varepsilon}(x,t)
&\le g(y,s)-2\varepsilon-M_1\rho(x)-M_2|t-s|\\
&\le g(x,t)+\omega(|x-y|)+\omega(|t-s|)-2\varepsilon-M_1\rho(x)-M_2|t-s|\\
&\le g(x,t).
\end{split}
\end{equation}
Thus, $\xi_1^{y,s,\varepsilon}$ satisfies the boundary condition and therefore we see that it is a subsolution of \eqref{e:master}-\eqref{e:initialboundary}.

Now, we define the function $\xi_1$ on $Q_T\cup\partial_pQ_T$ by
$$
\xi_1(x,t)=(\sup\{\xi_1^{y,s,\varepsilon}(x,t)\mid (y,s)\in\{0,l\}\times[0,T),\varepsilon>0\})^*.
$$
The uniformly continuity of $g$ implies that it is bounded in $\{0,l\}\times[0,T)$. Hence $\xi_1^{y,s,\varepsilon}\le\|g\|_\infty<+\infty$ in $Q_T\cup\partial_p Q_T$, that is, $\xi_1$ is bounded from above in $Q_T\cup\partial_pQ_T$.
Thus, Lemma \ref{l:stabsup} together with \eqref{e:bound} guarantee that $\xi_1$ is a subsolution of \eqref{e:master}-\eqref{e:initialboundary}.
Furthermore, \eqref{e:bound} and the fact that $\xi_1(x,t)\ge \sup_{\varepsilon>0}\xi_1^{x,t,\varepsilon}(x,t)=g(x,t)$ for $(x,t)\in\{0,l\}\times[0,T)$ implies that $\xi_1=g$ on $\{0,l\}\times[0,T)$.
Finally, since $\xi_1^{y,s,1}\le \xi_1\le g$ in $Q_T\cup\partial_pQ_T$ for each $(y,s)\in\{0,l\}\times[0,T)$ and $\xi^{y,s,1}_1$ is bounded in $Q_T\cup\partial_pQ_T$, then $\xi_1$ is bounded from below and, as a result , it is bounded in $Q_T\cup\partial_pQ_T$.
\end{proof}

\begin{proposition}\label{p:bottom}
Assume that $f$ is bounded and continuous, and $g$ is uniformly continuous.
Then, there are a bounded subsolution $\xi_2$ and a bounded supersolution $\eta_2$ of \eqref{e:master}-\eqref{e:initialboundary} that satisfy $\xi_2=\eta_2=g$ in $(0,l)\times\{0\}$.
\end{proposition}

\begin{proof}
For $y\in(0,l)$ and $\varepsilon>0$ we define $\xi_2^{y,\varepsilon}:Q_T\cup\partial_p Q_T\to\mathbb{R}$  by the following formula,
$$
\xi_2^{y,\varepsilon}(x,t)=g(y,0)-2\varepsilon-N_1\sigma^{{y}}(x)-(N_2+\|f\|_\infty)t.
$$
Here $N_1$ and $N_2$ are positive constants to be chosen later and 
\begin{equation}\label{e:sigma}
\sigma^{y}(x)=\frac{1}{\alpha}y^{1+\alpha}-\frac{1+\alpha}{\alpha}yx^{\alpha}+x^{1+\alpha}.
\end{equation}
Subsequently, we will suppress the superindex $y$.
We note that $\xi_2^{y,\varepsilon}\in\mathcal{C}(Q_{T,0})$.

We claim that $\xi_2^{y,\varepsilon}$ is a subsolution of \eqref{e:master}-\eqref{e:initialboundary} if $N_1$ and $N_2$ are taken large enough.
The direct computations using \eqref{e:wellknown} imply
\begin{align*}
(\xi_2^{y,\varepsilon})_t(x,t)-(D_x^\alpha \xi_2^{y,\varepsilon})_x(x,t)-f(x,t)
=-(N_2+\|f\|_\infty)-N_1\Gamma(2+\alpha)-f(x,t)\le0
\end{align*}
for all $(x,t)\in Q_T$. Hence,  the consistency result, see Proposition \ref{p:consistency},
implies that $\xi_2^{y,\varepsilon}$ is a viscosity subsolution of \eqref{e:master}.

We choose $N_1$ and $N_2$ so that
$$
N_1\ge\sup_{x\in[0,l]\setminus\{y\}}\frac{(\omega(|x-y|)-\varepsilon)^+}{\sigma(x)}\quad\text{and}\quad N_2\ge\sup_{t\in[0,T)}\frac{(\omega(t)-\varepsilon)^+}{t}.
$$
where $\omega$ is the continuity modulus of $g$. Then, 
for all $(x,t)\in\partial_pQ_T$ we have
\begin{align*}
\xi_2^{y,\varepsilon}(x,t)
&\le g(y,0)-2\varepsilon-N_1\sigma(x)-N_2t\\
&\le g(x,t)+\omega(|x-y|)+\omega(t)-2\varepsilon-N_1\sigma(x)-N_2t\\
&\le g(x,t).
\end{align*}
Therefore $\xi_2^{y,\varepsilon}$ is a subsolution of \eqref{e:master}-\eqref{e:initialboundary}.

We define the function $\xi_2$ on $Q_T\cup\partial_pQ_T$ by
$$
\xi_2(x,t)=(\sup\{\xi_2^{y,\varepsilon}(x,t)\mid y\in(0,l),\varepsilon>0\})^*
$$
It can be proved with the same idea as for $\xi_1$ in Proposition \ref{p:lateral} that $\xi_2=g$ in $(0,l)\times\{0\}$ and $\xi_2$ is bounded in $Q_T\cup\partial_pQ_T$, so we omit the details here.
\end{proof}

\begin{theorem}\label{t:existence}
Assume that $f$ is bounded and continuous, and $g$ is uniformly continuous.
Then, there exists a bounded solution $u\in C(Q_T\cup\partial_pQ_T)$ of \eqref{e:master}-\eqref{e:initialboundary}.
\end{theorem}

\begin{proof}
Let $\xi_1$ and $\xi_2$ be subsolutions of \eqref{e:master}-\eqref{e:initialboundary} from Propositions \ref{p:lateral} and \ref{p:bottom}, respectively.
Then, we easily see that $\xi=\xi_1\vee\xi_2$ is a bounded subsolution of \eqref{e:master}-\eqref{e:initialboundary} that satisfies $\xi=g$ on $\partial_pQ_T$.
Similarly, we have a bounded supersolution of the form $\eta:=\eta_1\wedge\eta_2$, which satisfies $\eta=g$ on $\partial_pQ_T$, where $\eta_1$ and $\eta_2$ are supersolutions given in Propositions \ref{p:lateral} and \ref{p:bottom}.
Theorem \ref{t:comparison} implies that $\xi\le\eta$ in $Q_T\cup\partial_pQ_T$.
Thus, by Lemma \ref{l:perron} we have a solution $u$ of \eqref{e:master}-\eqref{e:initialboundary} that satisfies $\xi\le u\le\eta$ in $Q_T\cup\partial_pQ_T$ and $u=g$ on $\partial_pQ_T$.
Using Theorem \ref{t:comparison} again, we see that $u^*\le u_*$ in $Q_T\cup\partial_pQ_T$, while the converse always holds.
Therefore $u$ is continuous in $Q_T\cup\partial_pQ_T$ and it
satisfies 
$$
\lim_{\delta\to0^+}\{|u(y,s)-g(x,t)|\mid(y,s)\in (Q_T\cup\partial_pQ_T)\cap\overline{B_\delta(x,t)}\}=0
$$
for each $(x,t)\in\partial_pQ_T$.
\end{proof}


\section{Stability}

The solution constructed in the previous section has a good stability property.
In this section we establish two typical results, one of which shows consistency with viscosity solution in the integer-order case.

\begin{theorem}\label{t:staborder}
Assume that $f$ is continuous.
For $\alpha\in(0,1)$ let $u_\alpha$ be a subsolution (resp. a supersolution) of \eqref{e:master} in which the fractional order is $\alpha$.
Let $\beta\in[0,1]$ and set
$$
\overline{u}_\beta=\limsup_{\alpha\to\beta,\alpha\neq\beta}{}^{*}u_\alpha\quad\text{(resp. $\underline{u}_\beta=\liminf_{\alpha\to\beta,\alpha\neq\beta}{}_{*}u_\alpha$).}
$$
If $\overline{u}_\beta<+\infty$ (resp. $\underline{u}_\beta>-\infty$), then $\overline{u}_\beta$  (resp. $\underline{u}_\beta$)
is a subsolution (resp. supersolution)  in $Q_{T,0}$ of \eqref{e:master} in which the fractional order is $\beta$. Here subsolutions $\overline{u}_0$ and $\overline{u}_1$ (resp. supersolutions $\underline{u}_0$ and $\underline{u}_1$) are in the usual viscosity sense.
\end{theorem}

In this theorem $\overline{u}_\beta$ and $\underline{u}_\beta$ stand for the upper half-relaxed limit and the lower half-relaxed limit, respectively. Namely,
\begin{align*}
\overline{u}_\beta(x,t)
&=(\limsup_{\alpha\to\beta,\alpha\neq\beta}{}^{*}u_\alpha)(x,t)\\
&=\lim_{\delta\to0^+}\sup\{u_\alpha(y,s)\mid
(y,s)\in Q_{T,0}\cap\overline{B_\delta(x,t)},\ 0<|\alpha-\beta|<\delta\}
\end{align*}
and $\underline{u}_\beta:=-\overline{(-u)}_\beta$.
Note that $\overline{u}_\beta:Q_{T,0}\to\mathbb{R}\cup\{+\infty\}$ is upper semicontinuous and $\underline{u}_\beta:Q_{T,0}\to\mathbb{R}\cup\{-\infty\}$ is lower semicontinuous; cf.  \cite[Lemma V.1.5]{BardiCapuzzoDolcetta}.

The proof of Theorem \ref{t:staborder} is based on two lemmas. Here is the first one:
\begin{lemma}\label{p:bascifact}
Let $\beta\in(0,1)$.
For $\alpha\in(0,1)$ let $u_\alpha:Q_{T,0}\to\mathbb{R}$ be an upper semicontinuous function.
Assume that $\overline{u}_\beta<+\infty$ in $Q_{T,0}$ and $\overline{u}_\beta-\varphi$ attains a strict maximum on $Q_T$ at $(\hat{x},\hat{t})\in Q_T$ for $\varphi\in C^{2,1}(Q_T)$.
Let $\delta\in(0,\hat{x})$ be a constant.
Then, there exists a neighborhood  $N\in\mathcal{N}_{\hat{x},\hat{t}}$ with the compact closure, sequences $(x_\varepsilon,t_\varepsilon)\in N$ and $\alpha_\varepsilon\in(0,1)\setminus\{\beta\}$ such that $(x_\varepsilon,t_\varepsilon)$ is a maximum point for $u_{\alpha_\varepsilon}^*-\varphi$ on $N$ and
$$
\lim_{\varepsilon\to0}(\alpha_\varepsilon,x_\varepsilon,t_\varepsilon,u_{\alpha_\varepsilon}^*(x_\varepsilon,t_\varepsilon))=(\beta,\hat{x},\hat{t},\overline{u}_\beta(\hat{x},\hat{t})).
$$ 
\end{lemma}

\begin{proof}
The proof is a trivial modification of \cite[Lemma V.1.6]{BardiCapuzzoDolcetta}.
\end{proof}

\begin{lemma}\label{p:convergence}
Let $\{f(\cdot;\lambda)\}_{\lambda\in \Lambda}\subset C^2[0,l]$, where $\Lambda$ is an index set.
Assume that 
$$\sup_{\lambda\in\Lambda}\|f'(\cdot;\lambda)\|_\infty<+\infty
$$ 
and $\sup_{\lambda\in\Lambda}|f'(x;\lambda)|\le Cx^{1+\nu}$ for a constant $C>0$ and $\nu>0$ as $x\to0^+$.
Then, 
$$
\lim_{\alpha\to 0^+}\sup_{(x,\lambda)\in [0,l]\times\Lambda}|(D_x^\alpha f)_x(x;\lambda)-f'(x;\lambda)|=0
$$
and
$$
\lim_{\alpha\to 1^-}\sup_{(x,\lambda)\in [0,l]\times\Lambda}|(D_x^\alpha f)_x(x;\lambda)-f''(x;\lambda)|=0.
$$
\end{lemma}

\begin{proof}
If $f$ does not depend on $\lambda$, this proposition follows easily from known facts.
Indeed, if we denote the Riemann-Liouville derivative by ${}^{RL}D_x^\alpha$, i.e., 
$$
{}^{RL}D_x^\alpha g(x)=\frac{1}{\Gamma(1-\alpha)}\frac{d}{dx}\int_0^x\frac{g(y)}{(x-y)^\alpha}dy
$$
for a function $g\in C^1(0,l)$, then $(D_x^\alpha f)_x(x)={}^{RL}D_x^\alpha {f'}(x)$, where $f(\cdot):=f(\cdot;\lambda)$.
Using the known formula (see \cite[Lemma 3.4]{Diethelm} for example) we have 
$$
{}^{RL}D_x^\alpha f'(x)=\frac{f'(0)}{\Gamma(1-\alpha)x^\alpha}+D_x^\alpha f'(x)=D_x^\alpha f'(x),
$$
where we used the fact $f'(0)=0$ by the assumption.
Let $J^{1-\alpha}$ denote the Riemann-Liouville integral:
$$
J^{1-\alpha}g(x)=\frac{1}{\Gamma(1-\alpha)}\int_0^x\frac{g(y)}{(x-y)^\alpha}dy
$$
for a given function $g$, where $J^0$ is the identity operator.
Then $D_x^\alpha f'(x)=J^{1-\alpha}f''(x)$.
According to \cite[Theorem 2.10]{Diethelm} we see that under the assumptions of this proposition, 
\begin{equation}\label{e:convergence1}
\lim_{\alpha\to0^+}\sup_{x\in[0,l]}|J^{1-\alpha}f''(x)-J^1f''(x)|=0
\end{equation}
and
\begin{equation}\label{e:convergence2}
\lim_{\alpha\to1^-}\sup_{x\in[0,l]}|J^{1-\alpha}f''(x)-J^0f''(x)|=0.
\end{equation}
Since $J^1f''(x)=f'(x)-f'(0)=f'(x)$ and $J^0f''(x)=f''(x)$, the conclusion in the case that $f$ does not depend on $\lambda$ turns out.

If $f$ depends on $\lambda$, then it is necessary to show that the convergence is also uniform in $\lambda$. Taking into account the above argument, it suffices to prove that the limits \eqref{e:convergence1} and \eqref{e:convergence2} are uniform in $\lambda$.
However, the proof is quite similar to that of \cite[Theorem 2.10]{Diethelm}. It exploits the fact that we have the bounds on $f'$, which are  uniform with respect to $\lambda$. We leave the details of the proof to the reader.
\end{proof}

\begin{proof}[Proof of Theorem \ref{t:staborder}]
Since the proof for supersolutions is similar, we give the proof for subsolutions.

Let $((\hat{x},\hat{t}),\varphi)\in Q_T\times C^{2,1}(Q_T)$ be such that $\bar{u}_\beta-\varphi$ attains a strict maximum on $Q_T$ at $(\hat{x},\hat{t})$. Fix $\delta\in(0,\hat{x})$ arbitrarily.
By Lemma \ref{p:bascifact} we have a neighborhood $N\in\mathcal{N}_{\hat{x},\hat{t}}$ 
with compact closure and $(\delta,\hat{t})\in N$,  sequences $(x_\varepsilon,t_\varepsilon)\in N$ and $\alpha_\varepsilon\in(0,1)\setminus\{\beta\}$ such that $(x_\varepsilon,t_\varepsilon)$ is a maximum point for $u_{\alpha_\varepsilon}^*-\varphi$ on $N$ and
\begin{equation}\label{e:stability0}
\lim_{\varepsilon\to0^+}(\alpha_\varepsilon,x_\varepsilon,t_\varepsilon,u_{\alpha_\varepsilon}^*(x_\varepsilon,t_\varepsilon))=(\beta,\hat{x},\hat{t},\overline{u}_\beta(\hat{x},\hat{t})).
\end{equation}
We may assume that $\delta<\inf_\varepsilon x_\varepsilon$.
Note that $u_{\alpha_\varepsilon}^*\le \bar{u}_\beta$ in $Q_{T,0}$ by definition of $\overline{u}_\beta$.

The case of $\beta\neq0,1$ is easy since we can use Lemma \ref{l:important}.
In fact, since $u_{\alpha_\varepsilon}$ is a subsolution of \eqref{e:master} with the fractional order $\alpha_\varepsilon$,
\begin{align*}
\varphi_t(x_\varepsilon,t_\varepsilon)\le& J^{\alpha_\varepsilon}[u_{\alpha_\varepsilon}^*,p_\varepsilon](x_\varepsilon,t_\varepsilon)+K_{(0,x_\varepsilon-\delta)}^{\alpha_\varepsilon}[\varphi,p_\varepsilon](x_\varepsilon,t_\varepsilon)\\
&+K_{(x_\varepsilon-\delta,x_\varepsilon)}^{\alpha_\varepsilon}[u_{\alpha_\varepsilon}^*,p_\varepsilon](x_\varepsilon,t_\varepsilon)+f(x_\varepsilon,t_\varepsilon),
\end{align*}
where $p_\varepsilon=\varphi_x(x_\varepsilon,t_\varepsilon)$.
Thus, we send $\varepsilon\to0$ to get
\begin{align*}
\varphi_t(\hat{x},\hat{t})\le& J^\beta[\bar{u}_\beta,p](\hat{x},\hat{t})+K_{(0,\hat{x}-\delta)}^\beta[\varphi,p](\hat{x},\hat{t})\\
&+K_{(\hat{x}-\delta,\hat{x})}^\beta[\bar{u}_\beta,p](\hat{x},\hat{t})+f(\hat{x},\hat{t})\quad\text{with $p=\varphi_x(\hat{x},\hat{t})$},
\end{align*}
which is the desired inequality.

Let $\beta=0$ or $\beta=1$.
Let $r>0$ be a constant such that $\overline{B_r(\hat{x},\hat{t})}\subset Q_T$.
We may assume that $(x_\varepsilon,t_\varepsilon)\in B_r(\hat{x},\hat{t})$ for all $\varepsilon$ because of \eqref{e:stability0}.
We may also take $\psi\in C^{2,1}(Q_T\cup\partial_pQ_T)$ that satisfies $\psi=\varphi$ in $B_r(\hat{x},\hat{t})$, $\max_{Q_{T,0}}(u_{\alpha_\varepsilon}-\psi)=(u_{\alpha_\varepsilon}-\psi)(x_\varepsilon,t_\varepsilon)$, $\sup_{t\in[\hat{t}-r,\hat{t}+r]}\|\psi_x(\cdot,t)\|_\infty<\infty$, and $\sup_{t\in[\hat{t}-r,\hat{t}+r]}|\psi_x(\cdot,t)|\le Cx^{1+\nu}$ for some $C>0$ and $\nu>0$ as $x\to0^+$.
Then we have
$$
\psi_t(x_\varepsilon,t_\varepsilon)\le (D_x^\alpha\psi)_x(x_\varepsilon,t_\varepsilon)+f(x_\varepsilon,t_\varepsilon).
$$
In order to pass to the limit with $\varepsilon\to0$ we invoke Proposition \ref{p:convergence} with $f(x;\lambda)=\psi(x,t)$ and $\Lambda=[\hat{t}-r,\hat{t}+r]$. This leads us to
\begin{equation}\label{e:stability1}
\psi_t(\hat{x},\hat{t})\le \psi_x(\hat{x},\hat{t})+f(\hat{x},\hat{t})\quad\text{if $\beta=0$}
\end{equation}
and
$$
\psi_t(\hat{x},\hat{t})\le \psi_{xx}(\hat{x},\hat{t})+f(\hat{x},\hat{t})\quad\text{if $\beta=1$.}
$$
These are desired inequalities since $\psi_t=\varphi_t$ and $\psi_{xx}=\varphi_{xx}$ at $(\hat{x},\hat{t})$.
\end{proof}

\begin{remark}
In the case of $\beta=0$, we usually take test functions from $C^1(Q_T)$, as a result one may think that the above proof is not complete. However, 
in the definition of viscosity solutions, we may use test functions 
having higher order derivatives; cf., e.g., \cite[Proposition 2.2.3]{Giga}. Therefore, we conclude that $\overline{u}_0$ is a subsolution after obtaining \eqref{e:convergence1}.
\end{remark}

\begin{proposition}
Let $u:Q_T\cup\partial_pQ_T\to\mathbb{R}$ be a solution of
\begin{equation}
\begin{cases}
u_t=(D_x^\alpha u)_x+f\quad&\text{in $Q_T$,}\\
u=g\quad&\text{on $\partial_pQ_T$,}
\end{cases}
\end{equation}
where $f\in C(Q_T)$ and $g\in C(\partial_pQ_T)$.
For $\varepsilon>0$ let $u_\varepsilon:Q_T\cup\partial_pQ_T\to\mathbb{R}$ be a solution of
\begin{equation}
\begin{cases}
(u_\varepsilon)_t=(D_x^\alpha u_\varepsilon)_x+f_\varepsilon\quad&\text{in $Q_T$,}\\
u_\varepsilon=g_\varepsilon\quad&\text{on $\partial_pQ_T$.}
\end{cases}
\end{equation}
Here $f_\varepsilon\in C(Q_T)$ and $g_\varepsilon\in C(\partial_pQ_T)$.
Assume that $\|g_\varepsilon-g\|_\infty\to0$ as $\varepsilon\to0$.
Then $\lim_{\varepsilon\to0}\|u_\varepsilon-u\|_\infty=0$.
\end{proposition}

\begin{proof}
Corollary \ref{c:contraction} implies that
$$
\sup_{Q_T\cup\partial_pQ_T}|u_\varepsilon-u|\le \sup_{\partial_pQ_T}|g_\varepsilon-g|+\int_0^T\sup_{x\in(0,l)}|(f_\varepsilon-f)(x,s)|ds.
$$
The right-hand side vanishes as $\varepsilon\to0$ so the conclusion is immediately obtained.
\end{proof}

\section{Regularity of solution}

In this section we study regularity by restricting the initial boundary condition $g$ to be Lipschitz continuous.
Let us denote the Lipschitz constant of $g$ by $L_g$ in what follows.

\begin{proposition}\label{p:lateralregularity}
Assume that $f$ is bounded and continuous, and $g$ is bounded Lipschitz continuous.
Let $u$ be the solution to \eqref{e:master}-\eqref{e:initialboundary}.
Then, there exist $L_1>0$ and $L_2>0$ such that for all $(x,t)\in Q_T\cup\partial_pQ_T$
$$
|u(x,t)-u(0,t)|\le L_1 x^\alpha
$$
and
$$
|u(x,t)-u(l,t)|\le L_2|l-x|.
$$
\end{proposition}

\begin{proof}
For $(y,s)\in\{0,l\}\times[0,T)$ we define $\xi^{y,s}:Q_T\cup\partial_p Q_T\to\mathbb{R}$  by the following formula
$$
\xi^{y,s}(x,t)=g(y,s)-((1+(c^y)^{-1})L_g+\|f\|_\infty)\rho(x)-L_g|t-s|.
$$
Here, $\rho$ is the same function as \eqref{e:rho} and
\begin{equation*}
c^{y}=\begin{cases}
	\rho(l)/l\quad&\text{when $y=0$,}\\
	\rho(0)/l\quad&\text{when $y=l$.}\\
\end{cases}
\end{equation*}
However, later on we will suppress the superscript $y$.

Using the method of the proof of Proposition \ref{p:lateral} we can show that $\xi^{y,s}$ is a subsolution of \eqref{e:master} and satisfies $\xi^{y,s}\le \|g\|_\infty<+\infty$, except for the difference of constants.
Moreover, since $\rho(x)\ge c|y-x|$ for $x\in[0,l]$, then we have
$$
\xi^{y,s}(x,t)\le g(y,s)-c^{-1}L_g\cdot c|y-x|-L_g|t-s|\le g(x,t)\qquad\hbox{for all }(x,t)\in\partial_p Q_T.
$$
Thus, we know that $(\sup_{(y,s)\in\{0,l\}\times[0,T)}\xi^{y,s})^*$ is a bounded subsolution of \eqref{e:master}-\eqref{e:initialboundary} by virtue of Lemma \ref{l:stabsup}.
Theorem \ref{t:comparison} yields $u\ge (\sup_{(y,s)\in\{0,l\}\times[0,T)}\xi^{y,s})^*$ in $Q_T\cup\partial_pQ_T$, so we get the estimate
\begin{align*}
u(x,t)\ge g(y,t)-((1+c^{-1})L_g+\|f\|_\infty)\rho(x)
\end{align*}
for all $(x,t)\in Q_T\cup\partial_pQ_T$ and $y\in\{0,l\}$.
Observing that
\begin{equation*}
\rho(x)\le\begin{cases}
	\displaystyle\frac{C}{\Gamma(1+\alpha)}x^\alpha\quad&\text{when $y=0$,}\\
	\displaystyle \frac{(1+\alpha)l^\alpha}{\Gamma(2+\alpha)}|l-x|\quad&\text{when $y=l$}
\end{cases}
\end{equation*}
for $x\in[0,l]$, we immediately obtain the one-side of the desired estimate.

Similarly, it can be proved that $\eta^{y,s}:Q_T\cup\partial_pQ_T\to\mathbb{R}$ defined by
$$
\eta^{y,s}(x,t)=g(y,s)+((1+c^{-1})L_g+\|f\|_\infty)\rho(x)+L_g|t-s|
$$
is a supersolution of \eqref{e:master}-\eqref{e:initialboundary} and satisfies $v^s\ge-\|g\|_\infty>-\infty$ in $Q_T\cup\partial_pQ_T$.
Thus $(\inf_{(y,s)\in\{0,l\}\times[0,T)}\eta^{y,s})_*$ is a bounded supersolution of \eqref{e:master}-\eqref{e:initialboundary} and a similar estimate yields the other side of the desired estimate.  
\end{proof}

\begin{proposition}\label{p:regularity}
Assume that $f$ is bounded and continuous, and $g$ is bounded and Lipschitz continuous.
Let $u$ be the solution to \eqref{e:master}-\eqref{e:initialboundary}. 
Then, for each $t\in[0,T)$, $u(\cdot,t)$ is locally Lipschitz continuous in $(0,l]$.
\end{proposition}

\begin{proof} To prove this proposition we follow the argument presented in \cite{BarlesChasseigneCiomagaImbert}, where the Ishii-Lions method \cite{IshiiLions} is extended to non-local equations. Fix $\hat{x}\in(0,l]$ arbitrarily. 

{\it Step 1} Given constants $L>0$, $C>0$ and $\eta>0$ we define
\begin{equation*}
\Phi_{L,C,\eta}(x,t,y):=u(x,t)-u(y,t)-L\phi(|x-y|)-C_1|x-\hat{x}|^2-\frac{2}{\eta(T-t)}
\end{equation*}
for all $(x,t,y)\in U:=[0,l]\times[0,T)\times[0,l]$, where $\phi$ is the concave function defined by
\begin{equation*}
	\phi(r)=\begin{cases}
		r-r^{1+\alpha}\quad&\text{($0\le r\le \hat{r}$),}\\
		{\hat{r}- \hat{r}^{1+\alpha}}\quad&\text{($r>\hat{r}$)}
	\end{cases}
\end{equation*} 
and $\hat{r}:=\argmax{}_{r\ge0}(r-r^{1+\alpha})=(\alpha+1)^{-1/\alpha}<1$. We claim that there is $C$ such that $\sup_U\Phi_{L,C,\eta}\le0$ for all large $L$ and all large $\eta$. Before showing this claim in Step 2, we will present its consequences. Namely, we see that for $(y,t)\in Q_T\cup\partial_pQ_T$ such that $|y-\hat{x}|\le\hat{r}$ we have
\begin{align*}
0&\ge \sup_U\Phi_{L,C,\eta}\\
&\ge \Phi_{L,C,\eta}(\hat{x},t,y)\\
&\ge u(\hat{x},t)-u(y,t)-L|\hat{x}-y|-\frac{2}{\eta(T-t)}
\end{align*}
since $\phi(r)\ge r$ for $r\in[0,\hat{r}]$. Then, after letting $\eta$ to infinity we obtain $u(y,t) \ge u(\hat x, t) - L|\hat x - y|$. If we interchange the role of $x$ and $y$, then we come to $u(\hat x,t) \ge u(y, t) - L|\hat x - y|$. This yields the Lipschitz continuity of $u(\cdot, t)$.

{\it Step 2} In order to show the claim made in Step 1, let us suppose the contrary: for all $C$, there is $L$ as large as we wish such that $\sup_U\Phi_{L,C,\eta}>0$.
In this case we first study maximum points of $\Phi_{L,C,\eta}$, which exist in $U$, because $u$ is bounded and $\Phi_{L,C,\eta}\to-\infty$ as $t\to T$. Let us denote a maximum point by $(\bar{x},\bar{t},\bar{y})$.
Since $\Phi_{L,C,\eta}(\bar{x},\bar{t},\bar{y})\ge\Phi_{L,C,\eta}(\hat{x},0,\hat{x})=-2/(\eta T)$, by rearranging the formula for $\Phi$ and considering large $\eta$, we have
\begin{equation}\label{e:sec7r1}
 L\phi(|\bar{x}-\bar{y}|)+C|\bar{x}-\hat{x}|^2\le u(\bar{x},\bar{t})-u(\bar{y},\bar{t})+\frac{2}{\eta T}\le 3\|u\|_\infty .
\end{equation}
Thus, we may assume that $\bar{x},\hat{y}\in(0,l]$ and $|\bar{x}-\bar{y}|\le\hat{r}$ by taking large $L$ and $C$. At this point it is good to recall also the choice of $\hat{x}$. Moreover, $\bar{x}\neq\bar{y}$, because $\Phi_{L,C,\eta}(\bar{x},\bar{t},\bar{x})>0$, but this is a contradiction. 

Let us suppose that $\bar{x}=l$. Since $\phi(|\bar{x}-\bar{y}|)\ge (1-\hat{r}^\alpha)|\bar{x}-\bar{y}|$ by the definition of $\phi$,  together with Proposition \ref{p:lateralregularity}, we have
\begin{align*}
\sup_U\Phi_{L,C,\eta}
&=u(l,\bar{t})-u(\bar{y},\bar{t})-L\phi(|l-\bar{y}|)-C|l-\hat{x}|^2-\frac{2}{\eta(T-\bar{t})}\\
&\le u(l,\bar{t})-u(\bar{y},\bar{t})-L\phi(|l-\bar{y}|)\\
&\le (L_2-L(1-\hat{r}^\alpha))|l-\bar{y}|,
\end{align*}
where $L_2$ is the same constant as in the statement of Proposition \ref{p:lateralregularity}. Thus $\bar{x}\neq l$ for suitably large $L$. It can be seen that $\bar{y}\neq l$ for the same reason. Furthermore, we also see $\bar{t}\neq0$ by arguing similarly using the Lipschitz continuity of $g$ instead of Proposition \ref{p:lateralregularity}. 

{\it Step 3} Given $\varepsilon>0$ we define
\begin{eqnarray*}
 &&\Phi_{L,C,\eta,\varepsilon}(x,t,y,s):=\\
 &&u(x,t)-u(y,t)-L\phi(|x-y|)-C|x-\hat{x}|^2-\frac{1}{\eta(T-t)}-\frac{1}{\eta(T-s)}-\frac{(t-s)^2}{\varepsilon}
\end{eqnarray*}
for $(x,t),(y,s)\in Q_T\cup\partial_pQ_T$.
There is a maximum point $(x_\varepsilon,t_\varepsilon,y_\varepsilon,s_\varepsilon)$ and it converges to a maximum point of $\Phi$ on $U$ by taking a subsequence if necessary.
We denote the limit by $(\bar{x},\bar{t},\bar{y},\bar{t})$ although it is not necessarily the same as the previous one.
Due to Step 2 we may assume that $(x_\varepsilon,t_\varepsilon),(y_\varepsilon,s_\varepsilon)\in Q_T$ and $0<|x_\varepsilon-y_\varepsilon|\le\hat{r}$ by considering suitably small $\varepsilon$.
Since $(x,t)\mapsto \Phi_{L,C,\eta,\varepsilon}(x,t,y_\varepsilon,s_\varepsilon)$ attains a maximum at $(x_\varepsilon,t_\varepsilon)$. Since it is sufficiently smooth we use it as test function, hence we have
$$
\frac{1}{\eta(T-t_\varepsilon)^2}+\frac{2(t_\varepsilon-s_\varepsilon)}{\varepsilon}\le J[u,p_\varepsilon+q_\varepsilon](x_\varepsilon,t_\varepsilon)+K_{(0,x_\varepsilon)}[u,p_\varepsilon+q_\varepsilon](x_\varepsilon,t_\varepsilon)+f(x_\varepsilon,t_\varepsilon),
$$
where $p_\varepsilon=L\phi'(|e_\varepsilon|)\hat{e}_\varepsilon$ and $q_\varepsilon=2C(x_\varepsilon-\hat{x})$.
Here and hereafter we write $e_\varepsilon=x_\varepsilon-y_\varepsilon$, $\hat{e}_\varepsilon=e_\varepsilon/|e_\varepsilon|$, $e=\bar{x}-\bar{y}$, and $\hat{e}=e/|e|$.
Similarly, since $(y,s)\mapsto -\Phi_{L,C,\eta,\varepsilon}(x_\varepsilon,t_\varepsilon,y,s)$ attains a minimum at $(y_\varepsilon,s_\varepsilon)$, we have
$$
-\frac{1}{\eta(T-s_\varepsilon)^2}+\frac{2(t_\varepsilon-s_\varepsilon)}{\varepsilon}\ge J[u,p_\varepsilon](y_\varepsilon,s_\varepsilon)+K_{(0,y_\varepsilon)}[u,p_\varepsilon](y_\varepsilon,s_\varepsilon)+f(y_\varepsilon,s_\varepsilon).
$$
Subtracting the second inequality from the first inequality yields
\begin{align*}
\frac{1}{\eta(T-t_\varepsilon)^2}+\frac{1}{\eta(T-s_\varepsilon)^2}
&\le J[u,p_\varepsilon+q_\varepsilon](x_\varepsilon,t_\varepsilon)-J[u,p_\varepsilon](y_\varepsilon,s_\varepsilon)\\
&\quad+K_{(0,x_\varepsilon)}[u,p_\varepsilon+q_\varepsilon](x_\varepsilon,t_\varepsilon)-K_{(0,y_\varepsilon)}[u,p_\varepsilon](y_\varepsilon,s_\varepsilon)\\
&\quad+f(x_\varepsilon,t_\varepsilon)-f(y_\varepsilon,s_\varepsilon).
\end{align*}
We claim that the right-hand side can be negative when choosing sufficiently large $L$ after sending $\varepsilon$ to $0$. This clearly gives a
contradiction.

{\it Step 4} We shall prove the claim made at the end of previous Step.
It is easy to see by a straightforward calculation that 
$$
\lim_{L\to\infty}\lim_{\varepsilon\to0}(J[u,p_\varepsilon+q_\varepsilon](x_\varepsilon,t_\varepsilon)-J[u,p_\varepsilon](y_\varepsilon,s_\varepsilon))=0.
$$
Thus, we only estimate the term $K_\varepsilon:=K_{(0,x_\varepsilon)}[u,p_\varepsilon+q_\varepsilon](x_\varepsilon,t_\varepsilon)-K_{(0,y_\varepsilon)}[u,p_\varepsilon](y_\varepsilon,s_\varepsilon)$ after splitting it into two expressions,
$$
K_{\varepsilon,1}=K_{(\delta,x_\varepsilon)}[u,p_\varepsilon+q_\varepsilon](x_\varepsilon,t_\varepsilon)-K_{(\delta,y_\varepsilon)}[u,p_\varepsilon](y_\varepsilon,s_\varepsilon)
$$
and
$$
K_{\varepsilon,2}=K_{(0,\delta)}[u,p_\varepsilon+q_\varepsilon](x_\varepsilon,t_\varepsilon)-K_{(0,\delta)}[u,p_\varepsilon](y_\varepsilon,s_\varepsilon).
$$
Here, $\delta$ is a constant such that $\sup_{\varepsilon}|e_\varepsilon|/2<\delta<\inf_{\varepsilon}|e_\varepsilon|$.
Notice that we may assume that $\sup_{\varepsilon}|e_\varepsilon|<\inf_{\varepsilon}(x_\varepsilon\wedge y_\varepsilon)$ for sufficiently large $L$. 

Thanks to the boundedness of $u$, the dominated convergence theorem is applicable to $K_{\varepsilon,1}$, giving
\begin{equation}\label{e:k11}
\lim_{\varepsilon\to0}K_{\varepsilon,1}=K_{(\delta,\bar{x})}[u,p+q](\bar{x},\bar{t})-K_{(\delta,\bar{y})}[u,p](\bar{y},\bar{t}),
\end{equation}
where $p=L\phi'(|e|)\hat{e}$ and $q=2C(\bar{x}-\hat{x})$.
We shall further estimate the right-hand side by dividing it into
$$
K_{1,1}=K_{(\delta,\bar{x}\wedge\bar{y})}[u,p+q](\bar{x},\bar{t})-K_{(\delta,\bar{x}\wedge\bar{y})}[u,p](\bar{y},\bar{t})
$$ 
and
$$
K_{1,2}=K_{(\bar{x}\wedge\bar{y},\bar{x})}[u,p+q](\bar{x},\bar{t})-K_{(\bar{x}\wedge\bar{y},\bar{y})}[u,p](\bar{y},\bar{t}).
$$ 
Since $\Phi_{L,C,\eta}(\bar{x},\bar{t},\bar{y})\ge\Phi_{L,C,\eta}(\bar{x}-z,\bar{t},\bar{y}-z)$, i.e.,
$$
u(\bar{x}-z,\bar{t})-u(\bar{x},\bar{t})-u(\bar{y}-z,\bar{t})+u(\bar{y},\bar{t})\le C(|\bar{x}-\hat{x}-z|^2-|\bar{x}-\hat{x}|^2)
$$
for all $z\in[0,\bar{x}\wedge\bar{y}]$, then we have
\begin{equation}\label{e:k12}
K_{1,1}\le C_\alpha\int_\delta^{\bar{x}\wedge\bar{y}}Cz^2\frac{dz}{z^{\alpha+2}}=\frac{C_\alpha C}{1-\alpha}((\bar{x}\wedge\bar{y})^{1-\alpha}-\delta^{1-\alpha}),
\end{equation}
where $C_\alpha=\alpha(\alpha+1)/\Gamma(1-\alpha)$.
Let us assume $\bar{x}<\bar{y}$ temporarily. Then, since $\Phi_{L,C,\eta}(\bar{x},\bar{t},\bar{y})\ge\Phi_{L,C,\eta}(\bar{x},\bar{t},\bar{y}-z)$, i.e,
$$
-u(\bar{y}-z,\bar{t})+u(\bar{y},\bar{t})\le L(\phi(|z+(\bar{x}-\bar{y})|)-\phi(|\bar{x}-\bar{y}|))
$$
for all $z\in[\bar{x},\bar{y}]$, we have
$$
K_{1,2}=-K_{(\bar{x},\bar{y})}[u,p](\bar y,\bar t)\le C_\alpha \int_{\bar{x}}^{\bar{y}}(L(\phi(|z+e|)-\phi(|e|))-pz)\frac{dz}{z^{\alpha+2}}.
$$
The monotonicity of $\phi$ implies that $\phi(|z+e|)\le \phi(z+|e|)$. Keeping this in mind, the concavity of $\phi$ implies that
$$
L(\phi(|z+e|)-\phi(|e|))-pz\le L\phi'(|e|)z-pz\le 2L\phi'(|e|)z
$$
and hence
\begin{equation}\label{e:k13}
K_{1,2}\le 2C_\alpha L\phi'(|e|)\int_{\bar{x}}^{\bar{y}}\frac{dz}{z^{\alpha+1}}=\frac{2C_\alpha L\phi'(|e|)}{\alpha}\left(\frac{1}{\bar{x}^\alpha}-\frac{1}{\bar{y}^\alpha}\right).
\end{equation}
Since we have $\bar{y}^\alpha - \bar{x}^\alpha\le |\bar{y}- \bar{x}|^\alpha$, then we obtain
$$
K_{1,2}\le 2C_\alpha L\phi'(|e|)\frac{|\bar{y}- \bar{x}|^\alpha}{\delta^{2 \alpha}}.
$$
We note that (\ref{e:sec7r1}) implies that $L|\bar{x}-\bar{y}| \le 3 \|u_0\|_\infty$.
Hence, we obtain the following bound on $K_{1,2}$,
$$
K_{1,2} \le 6C_\alpha L^{1-\alpha}  \|u_0\|_\infty^\alpha \delta^{-2 \alpha}.
$$
We also have
$$
\Phi_{L,C,\eta,\varepsilon}(x_\varepsilon,t_\varepsilon,y_\varepsilon,s_\varepsilon)\ge\Phi_{L,C,\eta,\varepsilon}(x_\varepsilon-z,t_\varepsilon,y_\varepsilon,s_\varepsilon),
$$
that is,
$$
u(x_\varepsilon-z,t_\varepsilon)-u(x_\varepsilon,t_\varepsilon)\le L(|\phi(|x_\varepsilon-y_\varepsilon-z|)-\phi(|x_\varepsilon-y_\varepsilon|))+C(|x_\varepsilon-\hat{x}-z|^2-|x_\varepsilon-\hat{x}|^2)
$$
and
$$
\Phi_{L,C,\eta,\varepsilon}(x_\varepsilon,t_\varepsilon,y_\varepsilon,s_\varepsilon)\ge\Phi_{L,C,\eta,\varepsilon}(x_\varepsilon,t_\varepsilon,y_\varepsilon-z,s_\varepsilon),
$$
that is,
$$
-u(y_\varepsilon-z,s_\varepsilon)+u(y_\varepsilon,s_\varepsilon)\le L(|\phi(|x_\varepsilon-y_\varepsilon+z|)-\phi(|x_\varepsilon-y_\varepsilon|))
$$
for all $z\in[0,\delta]$. Thus, it is readily seen that
$$
K_{\varepsilon,2}\le C_\alpha\int_0^\delta L(\phi(|z-e_\varepsilon|)+\phi(|z+e_\varepsilon|)-2\phi(|e_\varepsilon|))+Cz^2\frac{dz}{z^{\alpha+2}}.
$$
The fundamental theorem of the calculus and the definition of $\phi$ yield 
\begin{align*}
\phi(|z-e_\varepsilon|)+\phi(|z+e_\varepsilon|)-2\phi(|e_\varepsilon|)
&= - \alpha(1+\alpha) | e_\varepsilon|^{\alpha - 1} z^2
\int_0^1dt\int_{-1}^1\left(1+ \frac{zt\tau}{|e_\varepsilon|}\right)^{\alpha-1}t\,d\tau\\
&\le - \alpha(1+\alpha) 2^{\alpha -1}|e_\varepsilon|^{\alpha-1}.
\end{align*}
Therefore, since $\delta>\sup_{\varepsilon}|e_\varepsilon|/2$ we have
\begin{align*}
K_{\varepsilon,2}
&\le -2^{\alpha-1}\alpha(\alpha+1)C_\alpha L|e_\varepsilon|^{\alpha-1}\int_0^\delta\frac{dz}{z^\alpha}+C_\alpha C\int_0^\delta\frac{dz}{z^\alpha}\\
&\le \frac{-2^{\alpha-1}\alpha(\alpha+1)C_\alpha L|e_\varepsilon|^{\alpha-1}\delta^{1-\alpha}}{1-\alpha}+\frac{C_\alpha C\delta^{1-\alpha}}{1-\alpha}\\
&\le \frac{-2^{2(\alpha-1)}\alpha(\alpha+1)C_\alpha L}{1-\alpha}+\frac{C_\alpha C\delta^{1-\alpha}}{1-\alpha}.
\end{align*}
If we combine it with the bounds on $K_{1,1}$ and $K_{1,2}$, which are valid independently of $\epsilon$, then we deduce that $K$ is bounded above by a quantity, which tends to $-\infty$ as $L\to+\infty$ after $\varepsilon\to0$, and so the claim is now proved.
\end{proof}

\begin{proposition}\label{p:initial}
Assume that $f$ is bounded and continuous, and $g$ is bounded and Lipschitz continuous.
Let $u$ be the solution to \eqref{e:master}-\eqref{e:initialboundary}.
Then, there exists $L>0$ which depends only on $L_g$ and $f$ such that
$$
|u(x,t)-u(x,0)|\le Lt\quad\text{for all $(x,t)\in Q_T\cup\partial_p Q_T$.}
$$
\end{proposition}

\begin{proof}
For $y\in(0,l)$ and $\varepsilon>0$ let $\xi^{y,\varepsilon}:Q_T\cup\partial_pQ_T\to\mathbb{R}$ be defined by
$$
\xi^{y,\varepsilon}(x,t)=g(y,0)-2\varepsilon-N_1\sigma(x)-(L_g+N_2+\|f\|_\infty)t.
$$
Here $N_1$ and $N_2$ are positive constants to be chosen later and $\sigma$ is the same function introduced in the proof of Proposition \ref{p:bottom}.
Apart from the different constants, the proof that $\xi^{y,\varepsilon}$ is a subsolution of \eqref{e:master}-\eqref{e:initialboundary} and satisfies $\xi^{y,\varepsilon}\le\|g\|_\infty$ in $Q_T\cup\partial_pQ_T$ is the same as that of Proposition \ref{p:bottom}.
Thus, $\xi$ defined by
$$
\xi(x,t)=(\sup\{\xi^{y,\varepsilon}(x,t)\mid y\in(0,l),\varepsilon>0\})^*\quad\text{for $(x,t)\in Q_T\cup\partial_pQ_T$}
$$
is a subsolution of \eqref{e:master}-\eqref{e:initialboundary}.

We notice that $N_2$ can be taken to be independent of $y$ and $\varepsilon$. Recall that we take $N_2$ so that
$$
N_2\ge \sup_{t\in[0,T)}\frac{(L_g t-\varepsilon)^+}{t}.
$$
It is easily seen that the function $t\mapsto (L_g t-\varepsilon)^+/t$ is 
monotone increasing, hence its maximum is $(L_gT-\varepsilon)^+/T$, which is less than $L_g$. Thus, it is sufficient to take $N_2$ such that $N_2\ge L_g$. 

Theorem \ref{t:comparison} implies that $u\ge\xi$ in $Q_T\cup\partial_pQ_T$. Moreover, we have
\begin{align*}
\xi(x,t)&\ge \sup_{\varepsilon>0}(g(x,0)-2\varepsilon-(L_g+N_2+\|f\|_\infty)t)\\
&=g(x,0)-(L_g+\|f\|_\infty)t-\inf_{\varepsilon>0}2\varepsilon-N_2t\\
&=g(x,0)-(L_g+\|f\|_\infty+N_2)t.
\end{align*}
Therefore, the one-side of the desired inequality is established with $L=L_g+\|f\|_\infty+N_2$. 

Since it can proved similarly that
$$
\eta^{y,\varepsilon}(x,t)=g(y,0)+2\varepsilon+N_1\sigma(x)+(L_g+N_2+\|f\|_\infty)t,
$$
a function $\eta:=(\inf_{y\in(0,l),\varepsilon>0}\eta^{y,\varepsilon})_*$ is a supersolution of \eqref{e:master}-\eqref{e:initialboundary}.
Therefore, from a similar estimate as above, the other side of the desired inequality is also obtained immediately.
\end{proof}

\begin{proposition}\label{p:finalT}
Assume that $f$ is bounded and continuous, and $g$ is bounded and Lipschitz continuous.
Let $u$ be a unique solution to \eqref{e:master}-\eqref{e:initialboundary}. Then, there exists $L>0$ which depends only on $L_g$ and $f$ such that
$$
|u(x,t)-u(x,t+h)|\le L|h|
$$
for all $x\in[0,l],(t,h)\in[0,T)\times\mathbb{R}$ such that $t+h\in[0,T)$.
\end{proposition}

\begin{proof}
We will only show in the case $h>0$, the case $h<0$ is analogous.
Given a constant $L>0$ we define $v(x,t):=u(x,t+h)+Lh$. It is easy to see that $v$ is a supersolution of \eqref{e:master} in $(0,l)\times(0,T-h)$. Moreover, if $L$ is taken large enough, we have by Proposition \ref{p:initial}
$$
v(x,0)=u(x,h)+Lh\ge u(x,0)
$$
and by Lipchitz continuity of $g$, for $(x,t)\in\{0,l\}\times(0,T-h)$
$$
v(x,t)=g(x,t+h)+Lh\ge g(x,t)=u(x,t).
$$
Therefore we see that $u\le v$ on $\partial_pQ_{T-h}$. The comparison principle implies that $u\le v$ on $Q_{T-h}\cup\partial_pQ_{T-h}$, which is the one-side of the desired inequality.
The other-side is established by the similar argument for $w(x,t):=u(x,t+h)-Lh$.
\end{proof}

\section*{Acknowledgment}
The first author has worked on this paper at the University of Tokyo. A part of the research for this paper was done during the visit of the first author at the University of Warsaw, whose hospitality and support through the `Guest Program' is warmly acknowledged. The research of the second author was partially  supported by the National Science Center, Poland, through the grant number
2017/26/M/ST1/00700.

\providecommand{\bysame}{\leavevmode\hbox to3em{\hrulefill}\thinspace}
\providecommand{\MR}{\relax\ifhmode\unskip\space\fi MR }
\providecommand{\MRhref}[2]{%
  \href{http://www.ams.org/mathscinet-getitem?mr=#1}{#2}
}
\providecommand{\href}[2]{#2}

\end{document}